\documentclass[final,onefignum,onetabnum]{siamart220329}

\usepackage{braket,amsfonts}

\usepackage{array}

\usepackage[caption=false]{subfig}
\usepackage{pgfplots}

\newsiamthm{claim}{Claim}
\newsiamremark{remark}{Remark}
\newsiamremark{hypothesis}{Hypothesis}
\crefname{hypothesis}{Hypothesis}{Hypotheses}

\usepackage{algorithmic}

\usepackage{graphicx,epstopdf}
\usepackage{diagbox}
\usepackage{amssymb}
\usepackage{multirow}

\Crefname{ALC@unique}{Line}{Lines}

\usepackage{amsopn}

\usepackage{xspace}
\usepackage{bold-extra}
\usepackage[most]{tcolorbox}

\colorlet{texcscolor}{blue!50!black}
\colorlet{texemcolor}{red!70!black}
\colorlet{texpreamble}{red!70!black}
\colorlet{codebackground}{black!25!white!25}

\lstdefinestyle{siamlatex}{%
  style=tcblatex,
  texcsstyle=*\color{texcscolor},
  texcsstyle=[2]\color{texemcolor},
  keywordstyle=[2]\color{texemcolor},
  moretexcs={cref,Cref,maketitle,mathcal,text,headers,email,url},
}

\tcbset{%
  colframe=black!75!white!75,
  coltitle=white,
  colback=codebackground, 
  colbacklower=white, 
  fonttitle=\bfseries,
  arc=0pt,outer arc=0pt,
  top=1pt,bottom=1pt,left=1mm,right=1mm,middle=1mm,boxsep=1mm,
  leftrule=0.3mm,rightrule=0.3mm,toprule=0.3mm,bottomrule=0.3mm,
  listing options={style=siamlatex}
}

\newtcblisting[use counter=example]{example}[2][]{%
  title={Example~\thetcbcounter: #2},#1}

\newtcbinputlisting[use counter=example]{\examplefile}[3][]{%
  title={Example~\thetcbcounter: #2},listing file={#3},#1}

\DeclareTotalTCBox{\code}{ v O{} }
{ 
  fontupper=\ttfamily\color{black},
  nobeforeafter,
  tcbox raise base,
  colback=codebackground,colframe=white,
  top=0pt,bottom=0pt,left=0mm,right=0mm,
  leftrule=0pt,rightrule=0pt,toprule=0mm,bottomrule=0mm,
  boxsep=0.5mm,
  #2}{#1}

\patchcmd\newpage{\vfil}{}{}{}
\begin{tcbverbatimwrite}{tmp_\jobname_header.tex}
  \title{A Second-order method on graded meshes for fractional Laplacian via Riesz fractional derivative with a singular source term}

\author{Minghua Chen\thanks{School of Mathematics and Statistics, Gansu Key Laboratory of Applied Mathematics and Complex Systems, Lanzhou University, Lanzhou 730000, P.R. China (\email{chenmh@lzu.edu.cn}).}
\and Jianxing Han\thanks{School of Mathematics and Statistics, Gansu Key Laboratory of Applied Mathematics and Complex Systems, Lanzhou University, Lanzhou 730000, P.R. China (\email{hanjx2023@lzu.edu.cn}).}
\and Jiankang Shi\thanks{School of Mathematics and Computational Science, Xiangtan University, Xiangtan 411105, Hunan, People's Republic of China (\email{shijk@xtu.edu.cn}).}
\and Fan Yu\thanks{Institute for Math \& AI, Wuhan University, Wuhan 430072, P.R. China (\email{yufan24@whu.edu.cn}).}
}

\headers{Difference-quadrature for fractional Laplacian}{M. H. CHEN, J. X. HAN, J. K. SHI AND F. YU}
\end{tcbverbatimwrite}
\input{tmp_\jobname_header.tex}

\ifpdf
\hypersetup{ pdftitle={Guide to Using  SIAM'S \LaTeX\ Style} }
\fi

\begin{document}
\maketitle

\begin{tcbverbatimwrite}{tmp_\jobname_abstract.tex}
\begin{abstract}
  The high-order numerical analysis for fractional Laplacian  via the Riesz fractional derivative, under the low regularity solution,  has presented  significant challenges in the past decades.
  To fill in this gap, 
 we  design a grid mapping function  on graded meshes  to analyse the  local truncation errors, which  are  {\em far  less than} second-order convergence at the boundary layer. 
 To restore the  second-order global errors, we construct an appropriate  right-preconditioner for the resulting matrix algebraic equation. We prove that the proposed scheme  achieves  second-order convergence on graded meshes even if the source term is singular or hypersingular. Numerical experiments illustrate the theoretical results.
 The proposed approach is applicable for multidimensional fractional diffusion equations, gradient flows and nonlinear equations.
\end{abstract}


\end{tcbverbatimwrite}
\input{tmp_\jobname_abstract.tex}

\section{Introduction}
Fractional Laplacian is a powerful tool in modeling phenomena for anomalous diffusion, which appears naturally in the \(\alpha\)-stable L\'evy process instead of the standard Brownian motion 
\cite{ABBM2018,Bertoin:96,Getoor1961,MR2584076, MK:00}. It can be found in many applications, such as 
porous media flow \cite{MR2737788},
image processing \cite{MR3394445}, 
biophysics \cite{Andreu:10}.
In this work, we study a second-order difference-quadrature (DQ) scheme on graded meshes for the integral-differential version of the   fractional Laplacian  \cite{MR3941931}
\begin{equation} \label{eq:equation}
  \begin{cases}
    (-\Delta)^{\frac{\alpha}{2}} u(x) = f(x) & x \in \Omega,                    \\
    u(x) = 0                                 & x \in \mathbb{R} \setminus \Omega.
  \end{cases}
\end{equation}
Here  $(-\Delta)^{\frac{\alpha}{2}}$ is the integral-differential   fractional Laplacian, in terms of the Riesz (left and right Riemann-Liouville) fractional derivative \cite{ABBM2018, HuangO:14, MR4043885,MR2796453}, defined by  
\begin{equation} \label{def:operator}
  \begin{split}  
      (-\Delta)^{\frac{\alpha}{2}} u(x) 
      = -\frac{d^2}{dx^2} I^{2-\alpha} u(x) 
      \quad  \text{with} \quad 1<\alpha<2.
  \end{split}
\end{equation}
Note that the Riesz fractional integration can be realized in the form of the Riesz potential
\cite[eq.\,(1.30)]{FractionalDynamics}, namely,
\begin{equation} \label{def:I2-a}
  I^{2-\alpha} u(x) = \int_{\Omega} K(x-y) u(y) dy 
  \quad \text{with} \quad 
  K(x) = \frac{|x|^{1-\alpha}}{2\cos((2-\alpha)\pi/2)\Gamma(2-\alpha)}
  .
\end{equation}

The fractional Laplacian $\left(-\Delta\right)^{\frac{\alpha}{2}}$  can be defined in several equivalent ways \cite{MR3613319, MR4043885} on the whole space $ \mathbb{R}^n$. 
 For example, it can be defined  as a pseudo-differential operator  via the Fourier transform 
\[\mathcal{F}[\left(-\Delta\right)^{\frac{\alpha}{2}}u](\xi)=|\xi|^\alpha\mathcal{F}[u](\xi),\]
or in terms of the hypersingular integral operator 
\begin{gather} \label{def:IFLO}
(-\Delta)^{\frac{\alpha}{2}} u(x)=C_{n,\alpha} \,\,{\rm P.V.} \int_{\mathbb{R}^n}
		{\frac{u\left( x \right) -u\left( y \right)}{\left| x-y \right|^{n+\alpha}}dy}. 
\end{gather}

A key challenge in dealing with the fractional Laplacian arises from the fact that typical solutions $u$ exhibit a weak singularity at the boundary.
For example, the exact (Getoor) solution \cite{Getoor1961,HuangO:14,RosOtonSerra:14} is
\begin{equation*}\label{uf1}
		u(x)=\frac{2^{-\alpha} \varGamma \left( \frac{1}{2} \right)}{\varGamma \left( 1+\frac{\alpha}{2} \right) \varGamma \left( \frac{1+\alpha}{2} \right)} \left[ (x-a)\left( b-x \right) \right] ^{\frac{\alpha}{2}},
\end{equation*}
when $\Omega$ is a bounded interval $(a,b)\subset\mathbb{R}$ and $f\equiv 1$.
Moreover, the model equation \eqref{eq:equation}  can involve a singular/hypersingualr  source term, even if the exact solution $u$ is absolutely continuous \cite{MR4662768,MR4787529,MR4450105}.
These lead to a severe order reduction for many  numerical methods \cite{MR3941931,ChenFourth:14,Chen:14,HuangO:14}.



Among various techniques for approximating {\em integral} version of the fractional Laplacian \eqref{def:IFLO}, numerical quadrature with piecewise linear polynomials (collocation) is the simplest, since it only need a single integration and are much simpler to implement on a computer.
In \cite{HuangO:14}, Huang and Oberman first proposed a
quadrature-based finite difference method for solving the one-dimensional (1D) integral fractional Laplacian.
The method yields a numerical solution with an accuracy of
$\mathcal{O}\left(h^{2-\alpha}\right)$ in the 
discrete $L^\infty(\mathbb{R}^{n})$ norm, provided that the solution is  sufficiently smooth. 
However, this accuracy reduces to  $\mathcal{O}\left(h^{\alpha/2}\right)$ in the case $f\equiv 1$, leading to a boundary singularity in $u$.
Inspired by \cite{HuangO:14},  $\mathcal{O}\left(|\log h|h^{2-\alpha/2}\right)$ convergence  for $0<\alpha<2$ and  $\mathcal{O}\left(h^{\alpha}\right)$ for $\alpha\leq 4/3$, respectively, is proved  \cite{HW:22} in the discrete $L^{\infty}(\mathbb{R}^{n})$ norm on graded meshes for $n=1,2$ by means of a discrete barrier function. Recently,  $\mathcal{O}\left(h^{2-\alpha}\right)$ convergence for $0<\alpha< 1$ is given in \cite{Min:24} by collocation method on graded meshes, where it remains to be proved  for $1<\alpha<2$. It seems that achieving a second-order accurate scheme using piecewise linear polynomials collocation method for fractional Laplacian \eqref{def:IFLO} with $1<\alpha<2$ is not an easy task.

Nevertheless, there are already some important progress for numerically solving {\em integral-differential} version of the   fractional Laplacian  \eqref{def:operator} with $1<\alpha<2$   via the Riesz (left and right Riemann-Liouville) fractional  derivative. Take, for example,  the finite difference method \cite{C-N:12, ChenFourth:14, ChenSecondOdr:14, ChenHighOdr:15, DLIV:19, SLYL:21, SLAT:08, YLT:10, Sousa:15, Tian:15, ZLLBTA:14}, finite element method \cite{MR3941931,BTY:16, Deng:08, Ervin:18}, and spectral method \cite{DZZ:19, WMHK:19}. 
However, these methods may suffer from a severe order reduction when the exact  solution has a weak singularity at the boundary and the source term is singular.

How to design the second-order convergence for the model \eqref{eq:equation} under the low regularity solution has not been addressed in the literature. 
In this work, we combine the finite difference method and  numerical quadrature, called the difference-quadrature (DQ) method,    to approximate  the  {\em differential} and    {\em integral}   operator of  the   fractional Laplacian \eqref{def:operator} on graded meshes. 
The DQ method was proposed by the authors for solving Riesz fractional diffusion equations on uniform mesh \cite{ChenSuperlinearly:13,Chen:14} when the solution is sufficiently smooth with $u\in C^4(\bar{\Omega})$. 
However, the high-order numerical analysis \cite{ChenSuperlinearly:13,Chen:14}, under the low regularity solution, presents significant challenges in the past decades.
To fill in this gap, we  design a grid mapping function and a natural-skew  ordering  to handle  local truncation errors, and construct an appropriate  right-preconditioner for the resulting matrix algebraic equation. By utilizing the H\"older regularity of the data, we prove that the proposed scheme  is  second-order convergence on graded meshes even if the source term is hypersingular. 





\section{The main results}

In this section, we describe  the difference-quadrature scheme   on graded meshes  for fractional Laplacian \eqref{eq:equation}
via the  Riesz fractional derivative
and state our main results about the convergence rate of the numerical solutions.

\subsection{Difference-quadrature scheme}
\label{sec:numformat}

To keep the expressions simple below we assume we are on the interval $\Omega = (0, 2T)$, but everything can be shifted to an arbitrary interval $(a, b)$.
Partition $\Omega$ by the graded mesh
\begin{equation*}
    \pi_h : 0 = x_0 < x_1 < x_2 < \cdots < x_{2N-1} < x_{2N} = 2T ,
\end{equation*}
where we set
\begin{equation} \label{def:xj}
  x_j = \begin{cases}
    T \left(\frac{j}{N}\right)^r      & \text{for  } j=0, 1, ... , N,  \\
    2T - T \left(\frac{2N-j}{N}\right)^r  & \text{for  } j= N+1, N+2, ..., 2N,
  \end{cases}
\end{equation}
with a bounded grading exponent $r\ge 1$ .

Set
\(
  h_j = x_{j} - x_{j-1}
\) for $j = 1, 2, ... ,  2N$ 
and define $h := \frac{1}{N}$.
Let $S_h$ be the space of globally continuous piecewise linear functions on the mesh $\pi_h$
that vanish at $x = 0, 2T$. In this space, we choose as a basis the standard hat functions $\phi_j(x)$ in \cite{Min:24}.
Then, define the piecewise linear interpolant of the true solution \(u\) to be
\begin{equation} \label{def:interp}
  \Pi_hu(x) := \sum_{j=1}^{2N-1} u(x_j) \phi_j(x).
\end{equation}


We discretise \eqref{eq:equation} by replacing \(u(x)\) by a continuous piecewise linear function
$
  u_h(x) := \sum_{j=1}^{2N-1} u_j \phi_j (x) ,
$
whose nodal values \(u_j\) are to be determined by collocation at each mesh point \(x_i\) for \(i = 1, 2,...,2N - 1\):
\begin{equation} \label{def:discrete_equation}
  - D_h^{\alpha} u_h(x_i) := - D_h^2 I^{2-\alpha} u_h(x_i) = f(x_i) =: f_i 
\end{equation}
with the local truncation error 
\begin{equation} \label{def:truncation_error}
  \tau_i := -  D_h^{\alpha} \Pi_h u(x_i) - f(x_i)  \quad\text{for}\quad i = 1, 2, ... , 2N - 1 .
\end{equation}
Here the approximation of second order derivatives can be found in 
\cite[eq.\,(1.14)]{LeVequeFiniteDifference}
\begin{equation} \label{def:Dh2}
  D_h^2 u(x_i) := \frac{2}{h_i + h_{i+1}} \left( \frac{1}{h_i} u(x_{i-1}) - \left( \frac{1}{h_i} + \frac{1}{h_{i+1}} \right) u(x_i) + \frac{1}{h_{i+1}} u(x_{i+1}) \right).
\end{equation}
Moreover, the Riesz fractional derivatives in \eqref{def:operator} can be approximated by
\begin{equation} \label{eq:AU}
  -  D_h^{\alpha} u_h(x_i) 
  =  - D_h^2 I^{2-\alpha} \sum_{j=1}^{2N-1}u_j \phi_j (x_i) 
  = \sum_{j=1}^{2N-1} a_{ij} \, u_j .
\end{equation}
In particular, we have
\begin{equation} \label{eq:AUhat}    
- D_h^{\alpha} \Pi_h u(x_i) =  -\sum_{j=1}^{2N-1} D_h^{\alpha} \phi_j(x_i) \, u(x_j) = \sum_{j=1}^{2N-1} a_{ij} u(x_j).
\end{equation}

The discrete equation \eqref{def:discrete_equation} can be rewritten as the matrix form
\begin{equation} \label{eq:equation_matrix}
  AU = F ,
\end{equation}
where the coefficient matrix $A$ and the vectors $U$ and $F$ are defined by
$A=(a_{ij}) \in \mathbb{R}^{(2N-1) \times (2N-1)}$, $U=(u_1, \cdots, u_{2N-1})^T$ and $F=(f_1, \cdots, f_{2N-1})^T$.
In particular, the coefficient \(a_{ij}\) can be explicitly expressed as
\begin{equation} \label{eq:aij}
  \begin{aligned}
    a_{ij} &= -  D_h^2 I^{2-\alpha} \phi_j(x_i) \\
    &= - \frac{2}{h_{i} + h_{i+1}}
    \left( \frac{1}{h_{i}} \tilde{a}_{i-1,j} - \left( \frac{1}{h_{i}} + \frac{1}{h_{i+1}} \right) \tilde{a}_{i,j} +  \frac{1}{h_{i+1}} \tilde{a}_{i+1, j} \right) 
  \end{aligned}
\end{equation}
with the quadrature coefficients
\begin{equation*} \label{eq:tildeaij}
  \begin{aligned}
    \tilde{a}_{ij} &= I^{2-\alpha} \phi_j(x_i) \\
    &= \frac{\kappa_\alpha}{\Gamma(4-\alpha)}
    \left( \frac{|x_{i}-x_{j-1}|^{3-\alpha}}{h_{j}} -\left( \frac{1}{h_{j}} + \frac{1}{h_{j+1}} \right)|x_i-x_{j}|^{3-\alpha} +  \frac{|x_{i}-x_{j+1}|^{3-\alpha}}{h_{j+1}} \right) ,
  \end{aligned}
\end{equation*}
and $\kappa_\alpha = \frac{1}{2\cos((2-\alpha)\pi/2)} = -\frac{1}{2\cos(\alpha\pi/2)} > 0 $.

\subsection{Regularity of the true solution}
\label{sec:regularity}

For any \(\beta > 0\), 
we use the standard notation \(C^\beta(\bar{\Omega}),  C^\beta(\mathbb{R})\), etc., for Hölder spaces
and their norms and seminorms.
When no confusion is possible, 
we use the notation \(C^\beta (\Omega)\) to refer to \(C^{k,\beta'} (\Omega)\), where \(k\) is the greatest integer such that \(k<\beta\) and where \(\beta' = \beta - k\).
The Hölder spaces \(C^{k, \beta'}(\Omega)\) are defined as the subspaces of \(C^k(\Omega)\) consisting of functions whose \(k\)-th order partial derivatives are locally Hölder continuous \cite[p.\,52]{MR473443} with exponent \(\beta'\) in \(\Omega\).

For convenience, we define
\begin{equation} \label{def:delta}
  \delta(x) = \text{dist}(x, \partial \Omega) = \min\{x, 2T-x\}, \quad x\in (0, 2T),
\end{equation} 
and \( \delta(x, y) = \min\{\delta(x), \delta(y)\}\). 
To bound the derivatives of $u$,  we introduce the following \(\delta\)-dependent H\"older norms.
\begin{definition}[\(\delta\)-dependent H\"older norms \cite{ROSOTON2014275}]
  \label{def:delta-Holder-norm}
  For any $\beta > 0$, 
  write \(\beta = k + \beta'\), where \(k\) is an integer and \(\beta' \in (0, 1]\).
  Given \(\sigma\ge -\beta\), define the seminorm
  \begin{equation*}
    |w|_{\beta}^{(\sigma)} 
    = \sup_{x,y\in \Omega} \left( \delta(x,y)^{\beta+\sigma}\frac{|w^{(k)}(x)-w^{(k)}(y)|}{|x-y|^{\beta'}} \right).
  \end{equation*}
  The $\delta$-dependent H{\"o}lder norm \(\|\cdot\|_{\beta}^{(\sigma)}\) is defined as follows: 
  \begin{equation*}
    \|w\|_{\beta}^{(\sigma)} =
    \begin{cases}  
      \sum\limits_{l=0}^{k} \sup\limits_{x\in \Omega} \left( \delta(x)^{l+\sigma} |w^{(l)}(x)| \right) + |w|_{\beta}^{(\sigma)}  , & \sigma \ge 0, \\
      \|w\|_{C^{-\sigma}(\bar{\Omega})} + \sum\limits_{l=1}^{k} \sup\limits_{x\in \Omega} \left( \delta(x)^{l+\sigma} |w^{(l)} (x)| \right) + |w|_{\beta}^{(\sigma)}, & -1<\sigma < 0.
    \end{cases} 
  \end{equation*}
\end{definition}

\begin{lemma} \cite[pp.\,276-277]{ROSOTON2014275}
  \label{thm:Xacier1.2}
  Let \(f\in L^\infty(\Omega)\) and \(u\) be a solution of \eqref{eq:equation}.
  Then, \(u\in C^{\alpha/2}(\mathbb{R})\) and  \(u/\delta^{\alpha/2}\in C^\sigma(\bar{\Omega})\) for some \(\sigma \in (0, 1-\alpha/2)\), \(\alpha\in (1,2)\) with
  \begin{equation*}
    \|u\|_{C^{\alpha/2}(\mathbb{R})}\le C\|f\|_{L^\infty(\Omega)}
    \quad \text{and} \quad 
    \|u/\delta^{\alpha/2}\|_{C^\sigma(\bar{\Omega})} \le C \|f\|_{L^\infty(\Omega)}
  \end{equation*}
  for some positive constant \(C=C(\Omega, \alpha)\).
\end{lemma}

In particular, this result says that if \(f\in L^\infty(\Omega)\), then
\begin{equation} \label{eq:reg-u-0}
  |u(x)| \le C \delta(x)^{\alpha/2} \quad \text{for all } x\in \bar{\Omega}.
\end{equation}

\begin{lemma}\cite[Proposition 1.4]{ROSOTON2014275} \label{lmm:regularity}
    Let \(\beta>0\) be such that neither \(\beta\) nor \(\beta+\alpha\) is an integer. Let \(f \in C^\beta (\Omega)\) be such that \( \|f\|_{\beta}^{(\alpha/2)} < \infty\), and \(u \in C^{\alpha/2} (\mathbb{R})\) be a solution of \eqref{eq:equation}. Then, \(u \in C^{\beta+\alpha} (\Omega)\) and
    \begin{equation*}
      \|u\|_{\beta+\alpha}^{(-\alpha/2)} \le C \left( \|u\|_{C^{\alpha/2}(\mathbb{R})} + \|f\|_{\beta}^{(\alpha/2)} \right)
    \end{equation*}
    for some positive constant \(C=C(\Omega, \alpha, \beta)\).
\end{lemma}


By definition of \(\delta\)-dependent H\"older norms, we have following results obviously.
\begin{lemma} \label{lmm:regularity-u}
  Let $\beta=4-\alpha+\gamma$ with \(0 < \gamma <\alpha-1\).
  Assume that $f\in L^\infty(\Omega) \cap C^\beta(\Omega)$ be such that \( \|f\|_{\beta}^{(\alpha/2)} < \infty\),
  and $u$ be a solution of \eqref{eq:equation}. 
  Then 
  \begin{equation*}
  \begin{split}
    &|u^{(l)}(x)| \le C 
    \delta(x)^{\alpha/2-l}  \quad \text{for  } x \in \Omega \text{  and  } l=0, 1, 2, 3, 4,    \\
    &|f^{(l)}(x)| \le C
    \delta(x)^{-\alpha/2-l} \; \text{for  } x \in \Omega \text{  and  } l=0, 1, 2,
  \end{split}
  \end{equation*}
  for some positive constant \(C=C(\Omega, \alpha, \beta, f)\).
\end{lemma}
\begin{proof}
  Our hypotheses imply that \(2<\beta<3\), and \(4<\beta+\alpha<5\).
  By \Cref{lmm:regularity}, we have
  \begin{equation*}
    \|u\|_{\beta+\alpha}^{(-\alpha/2)} \le C \left( \|u\|_{C^{\alpha/2}(\mathbb{R})} + \|f\|_{\beta}^{(\alpha/2)}  \right).
  \end{equation*}
  By \Cref{def:delta-Holder-norm} and \Cref{thm:Xacier1.2}, it yields
  \begin{equation*}
    \sum_{l=1}^{4} \sup_{x\in \Omega} \left( \delta(x)^{l-\alpha/2} \left|u^{(l)}(x)\right| \right) \le C \left( \|f\|_{L^\infty(\Omega)} + \|f\|_{\beta}^{(\alpha/2)}  \right),
  \end{equation*}
  which is desired result \(l=1, 2, 3, 4\). 
  The case \(l = 0\) is covered by \eqref{eq:reg-u-0}.
  
  The second inequality can be obtained by \Cref{def:delta-Holder-norm}, namely,
  \begin{equation*}
    \sum_{l=0}^{2} \sup_{x\in \Omega} \left( \delta(x)^{l+\alpha/2} |f^{(l)}(x)| \right) \le \|f\|_{\beta}^{(\alpha/2)}.
  \end{equation*}
  The proof is completed.
\end{proof}

 

\subsection{Main results}
\label{sec:main}

The main results of this paper consist of the following theorems, which will be proved in \Cref{sec:proof-truncation-error} and \Cref{sec:proof-convergence}, respectively.

\begin{theorem}[Local Truncation Error] \label{thm:truncation-error}
Let $\alpha \in (1,2)$ and  $f\in L^\infty(\Omega) \cap C^\beta(\Omega)$ be such that \( \|f\|_{\beta}^{(\alpha/2)} < \infty\), where $\beta=4-\alpha+\gamma$ with \(0<\gamma<\alpha-1\).
Then,
\begin{equation*} 
  \begin{aligned}
    |\tau_i| &= | - D_h^{\alpha} \Pi_hu(x_i) - f(x_i) | \\
    &\le  C  h^{\min\{\frac{r\alpha}{2}, 2\}} \delta(x_i)^{-\alpha}
        + C(r-1) h^2 (T-\delta(x_{i}) + h_N)^{1-\alpha} , \quad 1\le i \le 2N-1
  \end{aligned}
\end{equation*}
for some positive constant \(C=C(\Omega, \alpha, \beta, r, f)\).
\end{theorem}

\begin{theorem}[Global Error]\label{thm:convergence}
Let $\alpha \in (1,2)$ and  $f\in L^\infty(\Omega) \cap C^\beta(\Omega)$ be such that \( \|f\|_{\beta}^{(\alpha/2)} < \infty\), where $\beta=4-\alpha+\gamma$ with \(0<\gamma<\alpha-1\).
Let \(u_i\) be the approximate solution of \(u(x_i )\) computed by the discretization scheme \eqref{def:discrete_equation}. Then,
\begin{equation*} 
  \max_{1\le i \le 2N-1} |u_i - u(x_i)| \le C h^{\min\{\frac{r\alpha}{2}, 2\}}
\end{equation*}
for some positive constant \(C=C(\Omega, \alpha, \beta, r, f)\).  
\end{theorem}

\section{Local Truncation Error}
\label{sec:proof-truncation-error}
For  convenience, we use the notation \( \simeq \)  ,
where \(x \simeq y\) 
means that \(c x \le y \le C x\) 
for some positive constants \(c\), \(C\) independent of \(h\).

For \(1\le j \le 2N\), we define the combination of adjacent grid points as
\begin{equation} \label{def:yjt}
  y_j^\theta = (1-\theta)x_{j-1} + \theta x_{j}, \quad \theta \in (0,1).
\end{equation}
Then, using the definition of grid points $\{x_j\}$ in \eqref{def:xj}, it follows that
\begin{lemma} \label{lmm:hilexi}
  Let $h=\frac{1}{N}$ and \(\delta(x_j)\) be defined by \eqref{def:delta}.
  Then we have
  \begin{equation*}
    h_j \simeq  h_{j+1} \simeq  h \delta(x_j)^{1-1/r}, \quad 1\le j \le 2N-1 ,
  \end{equation*}
  \[
    \delta(x_j) \simeq \delta(x_{j+1}) \simeq \delta(y_{j+1}^\theta),\;\, 1\le j \le 2N-2  .
  \]
\end{lemma}


We next give a detailed analysis of the local truncation error.


\subsection{Proof of \Cref{thm:truncation-error}}

The local truncation error \eqref{def:truncation_error} can be expressed
\begin{equation} \label{eq:truncerrordepart}
  \begin{aligned}
    \tau_i  &= -D_h^2 I^{2-\alpha} \Pi_h u(x_i) + \frac{d^2}{dx^2} I^{2-\alpha} u(x_i)   \\
     & = D_h^2 I^{2-\alpha} \left(u - \Pi_h u \right)(x_i) - \left( D_h^2 - \frac{d^2}{dx^2} \right) I^{2-\alpha} u(x_i)  .
  \end{aligned}
\end{equation}

We estimate each component of this partition.

  \begin{theorem} \label{lmm:trunerror2}
    There exists a constant $C$ such that
    \begin{equation*}
      \left| \left(D_h^2 - \frac{d^2}{dx^2}\right) I^{2-\alpha}u (x_i) \right| 
      \le C h^2 \delta(x_i)^{-\alpha/2-2/r} .
    \end{equation*}
  \end{theorem}
  \begin{proof}
    Since \(f\in C^2(\Omega)\) and
    \(
      - \frac{d^2}{dx^2} I^{2-\alpha}u(x) = f(x)
    \) for $x \in \Omega$,
    it implies \(I^{2-\alpha}u \in C^4(\Omega)\).
    From \Cref{lmm:Dh2simd2} in Appendix A,  we have for \(1\le i\le 2N-1\),
    \begin{equation*}
      \begin{aligned}
        &- \left(D_h^2 - \frac{d^2}{dx^2}\right) I^{2-\alpha}u (x_i)
         = \frac{h_{i+1}-h_{i}}{3} f'(x_i) \\
         & \quad + \frac{2}{h_i + h_{i+1}}\left(\frac{1}{h_i} \int_{x_{i-1}}^{x_{i}} f''(y) \frac{(y-x_{i-1})^3}{3!} dy + \frac{1}{h_{i+1}} \int_{x_{i}}^{x_{i+1}} f''(y) \frac{(y-x_{i+1})^3}{3!} dy\right).
      \end{aligned}
    \end{equation*}
    According to \Cref{lmm:hi1-hi,lmm:regularity-u,lmm:trucerr2d2f},  the desired result is obtained.
  \end{proof}

\label{subsec:Ri}

Now we consider the first term of the local truncation error in \eqref{eq:truncerrordepart}, which we denote for simplicity
\begin{equation} \label{eq:Ri}
  \begin{aligned}
    R_i & := D_h^2 I^{2-\alpha}(u-\Pi_h u)(x_i) , \quad 1\le i\le 2N-1 .
  \end{aligned}
\end{equation}
We have derived the following results concerning the estimation of $R_i$ including \Cref{thm:Ri-ilessN/2,thm:Ri-N/2le-i-leN}, which will be demonstrated  in \Cref{subsec:proofofRi}.
  \begin{theorem} \label{thm:Ri-ilessN/2}
    For \(1\le i < N/2\), there exists a constant $C$ such that
    \begin{equation*}
      |R_i| \le \begin{cases}
        C h^2 x_i^{-\alpha/2-2/r} ,             & \alpha/2 - 2/r + 1 > 0, \\
        C h^2 (x_i^{-1-\alpha}\ln(i) + \ln(N)), & \alpha/2 - 2/r + 1 = 0, \\
        C h^{r\alpha/2+r} x_i^{-1-\alpha},        & \alpha/2 - 2/r + 1 < 0.
      \end{cases}
    \end{equation*}
  \end{theorem}
  \begin{theorem} \label{thm:Ri-N/2le-i-leN}
    For \(N/2 \le i\le N\), there exists a constant $C$ such that
    \begin{equation*}
      |R_i| \le C(r-1) h^2 (T-x_{i} + h_N)^{1-\alpha}  + \begin{cases}
        C h^2,             & \alpha/2-2/r+1 > 0, \\
        C h^2 \ln(N) ,     & \alpha/2-2/r+1 = 0, \\
        C h^{r\alpha/2+r}, & \alpha/2-2/r+1 < 0.
      \end{cases}
    \end{equation*}
  \end{theorem}

\begin{remark} \label{rmk:symm}
Since the mesh \eqref{def:xj} is symmetric at $x = T$,
\Cref{thm:Ri-ilessN/2,thm:Ri-N/2le-i-leN} can be extended to the case of $3N/2 < i \le 2N-1$ and $N\le i \le 3N/2$, respectively.
\end{remark}

According to \Cref{lmm:trunerror2,thm:Ri-ilessN/2,thm:Ri-N/2le-i-leN,rmk:symm}, and using
\begin{gather*}
  h^2 x_i^{-\alpha/2-2/r} \le T^{\alpha/2-2/r} h^{\min\{\frac{r\alpha}{2}, 2\}} x_i^{-\alpha} ,\\
  h^{r\alpha/2+r} x_i^{-1-\alpha} \le T^{-1} h^{r\alpha/2} x_i^{-\alpha}, \\
  h^r x_i^{-1} \ln(i) = T^{-1} \frac{\ln(i)}{i^r} \le T^{-1}, 
  \quad h^r \ln(N) = \frac{\ln(N)}{N^r} \le 1,
\end{gather*}
the proof of \Cref{thm:truncation-error} is completed.

\subsection{Grid mapping functions}
\label{subsec:mesh-transport-functions}

In this subsection, 
we introduce the {\em natural-skew ordering} and {\em grid mapping functions}.
From \eqref{def:I2-a} and \eqref{eq:Ri}, we know that
\begin{equation} \label{eq:I2-au-Piu}
\begin{aligned}
    I^{2-\alpha} \left( u-\Pi_hu \right) (x_i)
     &= \sum_{j=1}^{2N} \int_{x_{j-1}}^{x_{j}} (u(y) - \Pi_hu(y)) K(x_i-y) dy
     = \sum_{j=1}^{2N} T_{ij}
\end{aligned}
\end{equation}
with
  \begin{equation} \label{def:Tij}
    T_{ij} = \int_{x_{j-1}}^{x_{j}} (u(y) - \Pi_hu(y)) K(x_i-y) dy, \quad i=0, \cdots ,2N,\; j=1, \cdots , 2N.
  \end{equation}
  To estimate $R_i$ more precisely, we define the {\em vertical difference quotients} of \(T_{ij}\)
  \begin{equation} \label{def:Vij}
    \begin{aligned}
      V_{ij} &=  \frac{2}{h_{i} + h_{i+1}}  \left( \frac{1}{h_{i}}  T_{i-1,j} - \left(\frac{1}{h_{i}} + \frac{1}{h_{i+1}}\right)  T_{i,j} + \frac{1}{h_{i+1}} T_{i+1,j} \right)  ,
    \end{aligned}
  \end{equation}
  and the {\em skew difference quotients} of \(T_{ij}\)
  \begin{equation} \label{def:Sij}
    S_{ij} =  \frac{2}{h_{i} + h_{i+1}}  \left( \frac{1}{h_{i}}  T_{i-1,j-1} - \left(\frac{1}{h_{i}} + \frac{1}{h_{i+1}}\right)  T_{i,j} + \frac{1}{h_{i+1}} T_{i+1,j+1} \right).
  \end{equation}

From \eqref{eq:Ri}, \eqref{eq:I2-au-Piu} and \eqref{def:Tij}, we have 
\begin{equation} \label{eq:departR1R2}
  R_1 = \sum_{j=1}^3 V_{1,j} + \sum_{j=4}^{2N} V_{1,j} \quad \text{and} \quad
  R_2 = \sum_{j=1}^4 V_{2,j} + \sum_{j=5}^{2N} V_{2,j}.
\end{equation}
Moreover, using \eqref{eq:Ri}-\eqref{def:Sij}, 
we can express $R_i$ based on the {natural-skew ordering}, as shown in  \Cref{fig:depart}:
\begin{equation} \label{eq:depart-Ri}
  \begin{aligned}
    R_i = & I_1 + I_2 + I_3 + I_4 + I_5   \quad \text{for} \quad   3\le i \le N .
  \end{aligned}
\end{equation}
Here, 
$ I_1 =  \sum_{j=1}^{k-1} V_{ij}$,
$I_3 = \sum_{j=k+1}^{m-1} S_{ij}$,
$I_5 =  \sum_{j=m+1}^{2N} V_{ij}$,
and
\begin{gather*}
    I_2 =  \frac{2}{h_i + h_{i+1}}
    \left( \frac{1}{h_{i+1}} (T_{i+1, k} +  T_{i+1, k+1})
    - (\frac{1}{h_{i}}+\frac{1}{h_{i+1}}) T_{i,k} \right) ,  \\
    I_4 =  \frac{2}{h_i + h_{i+1}}
    \left( \frac{1}{h_{i}} (T_{i-1, m} +  T_{i-1, m-1})
    - (\frac{1}{h_{i}}+\frac{1}{h_{i+1}}) T_{i,m} \right)
\end{gather*}
with
\begin{equation} \label{def:m2c}
  k=\lceil i/2\rceil ,\quad
m= \begin{cases}    
2i, & 3\le i < N/2, \\ 2N-\lceil N/2 \rceil + 1, & N/2 \le i \le N.
\end{cases}
\end{equation}

Note that \(I_1\) and \(I_5\) along with $V_{ij}$ as defined in \eqref{def:Vij}, represent natural (vertical) ordering, while \(I_3\), along with $S_{ij}$ as defined in \eqref{def:Sij}, represents skew ordering.

\begin{figure}[htbp]
  \centering 
  \includegraphics[width=0.7\textwidth]{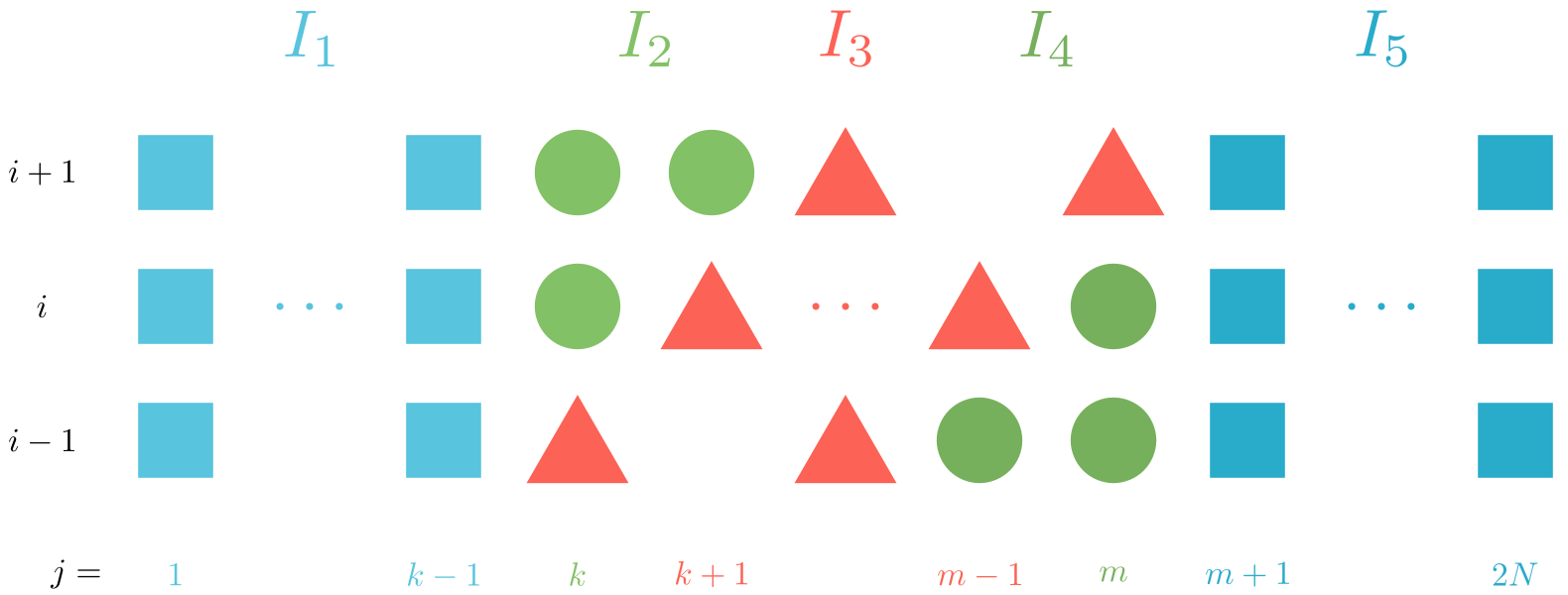}
  \caption{Natural-Skew ordering of \(R_i\).}
  \label{fig:depart}
\end{figure}

The complexity in estimating \(S_{ij}\) in \eqref{def:Sij} lies in the fact that the integral domains for  \(T_{i-1, j-1}, T_{i,j}\) and \(T_{i+1,j+1}\) are distinct.
We first normalize $T_{ij}$ to the unit interval.

\begin{lemma} \label{lmm:Tij-normalized}
  For any \(y\in (x_{j-1}, x_{j})\), there exits 
  \begin{equation*} 
    \begin{aligned}
      T_{ij} & = \int_{x_{j-1}}^{x_{j}} (u(y) - \Pi_hu(y)) K(x_i - y) dy  \\
             & = \int_{0}^{1}
      -\frac{\theta (1-\theta)}{2} h_j^3 u''(y_j^\theta) K(x_i - y_j^\theta) d\theta   \\
             & \quad +  \int_{0}^{1} \frac{\theta (1-\theta)}{3!} h_j^4  K(x_i - y_j^\theta) \left( \theta^2 u'''(\eta_{j1}^\theta) - (1-\theta)^2 u'''(\eta_{j2}^\theta) \right) d\theta
    \end{aligned}
  \end{equation*}
  with \(\eta_{j1}^\theta \in (x_{j-1}, y_j^\theta), \eta_{j2}^\theta \in (y_j^\theta, x_j)\).
\end{lemma}
\begin{proof}
  By \eqref{def:Tij} and \Cref{lmm:Dyj}, the desired result is obtained.
\end{proof}

To estimate the local truncation error more concisely, we construct the following grid mapping functions.

\begin{definition} 
\label{def:gridmapfunc}
  For \(1\le i, j \le 2N-1\), we define the grid mapping functions
  \begin{equation} \label{def:yij}
    y_{i,j}(x) = \begin{cases}
      (x^{1/r} + Z_{j-i})^r               & i< N, j< N, \\
      \dfrac{x^{1/r} - Z_i}{Z_1} h_N + x_N  & i< N, j=N, \\
      2T - (Z_{2N-(j-i)} - x^{1/r})^r     & i< N, j>N, \\
      \left(\dfrac{Z_1}{h_N}  (x-x_N) + Z_j \right)^r  & i=N, j< N, \\
      x                                   & i=N, j=N, \\
    \end{cases}
  \end{equation}
  and
  \begin{equation*} 
    y_{i,j}(x) = \begin{cases}
      2T - \left(\dfrac{Z_1}{h_N}  (2T-x-x_N) + Z_{2N-j} \right)^r   & i=N , j > N, \\
      (Z_{2N+j-i} - (2T - x)^{1/r})^r  & i > N, j< N, \\
      \dfrac{Z_{2N-j} - (2T-x)^{1/r}}{Z_1} h_N + x_N  & i > N, j=N, \\
      2T-((2T-x)^{1/r}-Z_{j-i})^r & i > N, j> N \\
    \end{cases}
  \end{equation*}
  with  \( Z_{j} := T^{1/r}\frac{j}{N} \).
\end{definition}

Let us further define 
  \begin{equation} \label{def:hij}
    h_{i,j}(x) = y_{i,j}(x) - y_{i,j-1}(x),
  \end{equation}
  \begin{equation} \label{def:yijt}
    y_{i,j}^\theta(x) = (1-\theta) y_{i,j-1}(x) + \theta y_{i,j-1}(x), \quad \theta \in (0, 1),
  \end{equation}
  \begin{equation} \label{def:Pj-itheta-jlN}
    {P_{i,j}^\theta}(x) = ({h_{i,j}}(x))^3  K(x - {y_{i,j}^\theta}(x) ) u''({y_{i,j}^\theta}(x)),
  \end{equation}
  \begin{equation} \label{def:Qj-itheta-jlN}
    {Q_{i,j,l}^{\theta}}(x) = ({h_{i,j}}(x))^l K(x - {y_{i,j}^\theta}(x)), \quad l=3, 4.
  \end{equation}
Then, we can check that
\begin{equation} \label{eq:prop-of-GMFs}
    \begin{gathered}
      y_{i,j}(x_{i-1}) = x_{j-1}, \quad y_{i,j}(x_{i}) = x_{j}, \quad y_{i,j}(x_{i+1}) = x_{j+1}, \\
      h_{i,j}(x_{i-1}) = h_{j-1}, \quad h_{i,j}(x_{i}) = h_j, \quad h_{i,j}(x_{i+1}) = h_{j+1}, \\
      y_{i,j}^\theta(x_{i-1}) = y_{j-1}^\theta, \quad y_{i,j}^\theta(x_{i}) = y_{j}^\theta, \quad y_{i,j}^\theta(x_{i+1}) = y_{j+1}^\theta.
    \end{gathered}
\end{equation}
From \eqref{def:Pj-itheta-jlN}, \eqref{def:Qj-itheta-jlN} and \Cref{lmm:Tij-normalized}, we can rewrite \(T_{ij}\) as
  \begin{equation} \label{lmm:Tij-express-as-int-of-function}
    \begin{aligned}
      T_{ij} & = \int_{0}^{1} -\frac{\theta (1-\theta)}{2} {P_{i,j}^\theta}(x_i) d\theta    \\
             & \quad + \int_{0}^{1} \frac{\theta (1-\theta)}{3!}{Q_{i,j,4}^\theta}(x_i) \left[ \theta^2  u'''(\eta_{j,1}^\theta) - (1-\theta)^2 u'''(\eta_{j,2}^\theta) \right] d\theta.
    \end{aligned}
  \end{equation}
  From \eqref{def:Dh2}, \eqref{def:Sij} and \eqref{lmm:Tij-express-as-int-of-function}, for \( 1 \le i \le 2N-1\), \(2\le j \le 2N-1\), we have
  \begin{equation} \label{eq:Sij-int}
    \begin{aligned}
      S_{ij}
             & = \int_{0}^{1} -\frac{\theta (1-\theta)}{2} D_h^2 P_{i,j}^\theta(x_i)  d\theta    \\
             & \quad +  \int_{0}^1 \frac{\theta^3 (1-\theta)}{3!} \frac{2}{h_{i} + h_{i+1}}\left( \frac{{Q_{i,j,4}^\theta}(x_{i+1}) u'''(\eta_{j+1,1}^\theta) - {Q_{i,j,4}^\theta}(x_{i}) u'''(\eta_{j,1}^\theta)}{h_{i+1}}\right)  d\theta \\
             & \quad -  \int_{0}^1 \frac{\theta^3 (1-\theta)}{3!} \frac{2}{h_{i} + h_{i+1}}\left( \frac{{Q_{i,j,4}^\theta}(x_{i}) u'''(\eta_{j,1}^\theta) - {Q_{i,j,4}^\theta}(x_{i-1}) u'''(\eta_{j-1,1}^\theta)}{h_{i}}\right)  d\theta   \\
             & \quad -  \int_{0}^1 \frac{\theta (1-\theta)^3}{3!} \frac{2}{h_{i} + h_{i+1}}\left( \frac{{Q_{i,j,4}^\theta}(x_{i+1}) u'''(\eta_{j+1,2}^\theta) - {Q_{i,j,4}^\theta}(x_{i}) u'''(\eta_{j,2}^\theta)}{h_{i+1}}\right)  d\theta \\
             & \quad +  \int_{0}^1 \frac{\theta (1-\theta)^3}{3!} \frac{2}{h_{i} + h_{i+1}}\left( \frac{{Q_{i,j,4}^\theta}(x_{i}) u'''(\eta_{j,2}^\theta) - {Q_{i,j,4}^\theta}(x_{i-1}) u'''(\eta_{j-1,2}^\theta)}{h_{i}}\right)  d\theta.
    \end{aligned}
  \end{equation}


The following \Cref{lmm:gen-prop-of-MTFs,lmm:esitmate-of-MTFs-1,lmm:esitmate-of-MTFs-2,lmm:d2Pj-itle,lmm:dQj-itle} about the grid mapping functions will be used in next subsection,
which are proved in \Cref{sec:prfs-of-GMFs}.

\begin{lemma} \label{lmm:gen-prop-of-MTFs}
  For any \(\xi \in (x_{i-1}, x_{i+1})\), \(2 \le i,j \le 2N-2\), there exist
  \begin{gather*}
    \xi \simeq x_i, \quad \delta(y_{i,j}(\xi))\simeq \delta(x_j), \quad h_{i,j}(\xi) \simeq h_j, \\
    |y_{i,j}(\xi) - \xi| \simeq |x_j - x_i|, \quad |y_{i,j-1}(\xi) - \xi| \simeq |x_{j-1} - x_i|, \\
    |y_{i,j}^\theta(\xi) - \xi| = (1-\theta)|y_{i,j-1}(\xi) - \xi| + \theta |y_{i,j}(\xi) - \xi| \simeq |y_j^\theta - x_i|.
  \end{gather*}
\end{lemma}

\begin{lemma} \label{lmm:esitmate-of-MTFs-1}
  For any \(\xi \in (x_{i-1}, x_{i+1}) \), \(2\le i \le N, 2\le j \le 2N-2\), there exist
  \begin{gather*}
    |h_{i,j}'(\xi)| \le C (r-1) Z_1 x_i^{1/r-1}  \delta(x_j)^{1-2/r}
    \le C(r-1) h_j x_i^{1/r-1} \delta(x_j)^{-1/r}, \\
    \left|(y_{i,j}(\xi) - \xi)'\right| \le C x_i^{-1} |x_j - x_i|.
  \end{gather*}
\end{lemma}

\begin{lemma} \label{lmm:esitmate-of-MTFs-2}
  For any \(\xi \in (x_{i-1}, x_{i+1})\), \(2\le i \le N, 2\le j \le 2N-2\), there exist
  \begin{equation*}
    |y_{i,j}''(\xi)| \le C(r-1)
    \begin{cases}
      x_j^{-1/r} x_i^{1/r-2} |x_j - x_i|  ,& i < N, j < N,  \\
      x_N^{1-1/r} x_i^{1/r-2}             ,& i < N, j = N , \\
      \delta(x_j)^{1-2/r} x_i^{1/r-2} x_N^{1/r}     ,& i < N, j > N , \\
      \delta(x_j)^{1-2/r} x_N^{2/r-2}             ,& i = N, j \neq N , \\
      0   & i = N, j = N.
    \end{cases}
  \end{equation*}
  For \(2\le i \le N, 3\le j \le 2N-2\),  there exist 
  \begin{equation*}
    |h_{i,j}''(\xi)| \le C (r-1)
    \begin{cases}
      Z_1 x_i^{1/r-2} x_j^{-2/r} (|x_j - x_i| + x_j)   ,& i < N, j < N , \\
      x_i^{1/r-2} x_N^{1-1/r}                        ,& i < N, j=N, N+1, \\
      Z_1 x_i^{1/r-2} \delta(x_j)^{1-3/r}  x_N^{1/r}    ,& i < N, j>N+1 ,\\
      Z_1 x_N^{2/r-2} \delta(x_j)^{1-3/r}               ,& i = N, j \neq N, N+1 , \\
      x_N^{-1}                                       ,& i = N, j = N, N+1.
    \end{cases}
  \end{equation*}
\end{lemma}

  \begin{lemma} \label{lmm:d2Pj-itle}
    Let \({P_{i,j}^\theta}(x_i)\) be defined by \eqref{def:Pj-itheta-jlN} and the difference quotient operator $D_h^2$ be defined by \eqref{def:Dh2}. Then we have
    \paragraph{Case 1}
    For \(3\le i < N, \lceil\frac{i}{2}\rceil+1 \le j \le \min\{2i-1, N-1\}\), there exists
    \begin{equation*}
      |D_h^2 {P_{i,j}^\theta}(x_{i})| \le C h_j^3 |y_j^\theta - x_i|^{1-\alpha} x_i^{\alpha/2-4}.
    \end{equation*}
    \paragraph{Case 2}
    For \(N/2\le i \le N\), \(j=N, N+1\), there exists
    \begin{equation*}
      \begin{aligned}
        |D_h^2 P_{i,j}^\theta(\xi)| 
        \le C h_j^3 |y_j^\theta - x_i|^{1-\alpha}  + C(r-1) h_j^2 \Big(|y_j^\theta - x_i|^{1-\alpha} + h_j|y_j^\theta - x_i|^{-\alpha}
        \Big).
      \end{aligned}
    \end{equation*}
    \paragraph{Case 3}
    For \(N/2 \le i \le N\), \(N+2 \le j \le 2N-\lceil\frac{N}{2}\rceil\), there exists
    \begin{equation*}
      \begin{aligned}
        |D_h^2 P_{i,j}^\theta(\xi)| 
        \le C h_j^3 \Big(|y_j^\theta - x_i|^{1-\alpha}  + (r-1) |y_j^\theta - x_i|^{-\alpha}
        \Big).
      \end{aligned}
    \end{equation*}
  \end{lemma}

\begin{lemma} \label{lmm:dQj-itle}
  Let \({Q_{i,j, l}^\theta}(x_i)\) be defined by \eqref{def:Qj-itheta-jlN}. Then we have
  for \(2 \le i \le N\), \(2 \le j \le 2N-2\), \(l=3,4\), there exist
  \begin{equation*}
    \begin{aligned}
      & \left| \frac{{Q_{i,j,l}^\theta}(x_{i+1}) u^{(l-1)}(\eta_{j+1}^\theta) - {Q_{i,j,l}^\theta}(x_{i}) u^{(l-1)}(\eta_{j}^\theta)}{h_{i+1}}\right|  \\
      & \quad \le C h_j^l |y_j^\theta - x_i|^{1-\alpha} x_i^{-1} \delta(x_j)^{\alpha/2-l+1-1/r} (x_i^{1/r} + \delta(x_j)^{1/r}),
    \end{aligned}
  \end{equation*}
  and
  \begin{equation*}
    \begin{aligned}
      & \left| \frac{{Q_{i,j,l}^\theta}(x_{i}) u^{(l-1)}(\eta_{j}^\theta) - {Q_{i,j,l}^\theta}(x_{i-1}) u^{(l-1)}(\eta_{j-1}^\theta)}{h_{i}} \right|  \\
      & \quad \le C h_j^l |y_j^\theta - x_i|^{1-\alpha} x_i^{-1} \delta(x_j)^{\alpha/2-l+1-1/r} (x_i^{1/r} + \delta(x_j)^{1/r})
    \end{aligned}
  \end{equation*}
  with \(\eta_{j}^\theta \in (x_{j-1}, x_{j})\).
\end{lemma}

\subsection{Error analysis of $R_i$}
\label{subsec:proofofRi}
In this subsection, we estimate the first term of the local truncation error $R_i$ in \eqref{eq:Ri} through \eqref{eq:departR1R2} and \eqref{eq:depart-Ri}.
We denote
\begin{equation} \label{def:Kyx}
    K_y(x) := K(x-y) = \frac{\kappa_\alpha}{\Gamma(2-\alpha)} |x - y|^{1-\alpha}, \quad 1<\alpha<2,
\end{equation}
where the kernel function $K(x)$ is given in \eqref{def:I2-a} and $\kappa_\alpha$ is given in \eqref{eq:aij}.

\begin{lemma} \label{lmm:Ri-I5-1}
Let \(I_5 = \sum_{j=m+1}^{2N} V_{ij}\) be defined by \eqref{eq:depart-Ri}. Then we have
\paragraph{Case 1}
For  \(1 \le i < N/2\) and \(m=\max\{2i, 3\}\), there exists
\begin{equation*}
  \begin{aligned}
    \sum_{j=m+1}^{2N} \left| V_{ij} \right| 
    \le C h^2 x_i^{-\alpha/2-2/r}
    + \begin{cases}
          C h^2,             & \alpha/2-2/r+1 > 0, \\
          C h^2 \ln(N) ,     & \alpha/2-2/r+1 = 0, \\
          C h^{r\alpha/2+r}, & \alpha/2-2/r+1 < 0.
        \end{cases}
  \end{aligned}
\end{equation*}
\paragraph{Case 2}
For \(N/2 \le i \le N\) and $m=2N-\lceil \frac{N}{2} \rceil +1$, there exists
\begin{equation*}
  \begin{aligned}
    \sum_{j=m+1}^{2N} |V_{ij}|
    \le \begin{cases}
          C h^2,             & \alpha/2-2/r+1 > 0, \\
          C h^2 \ln(N) ,     & \alpha/2-2/r+1 = 0, \\
          C h^{r\alpha/2+r}, & \alpha/2-2/r+1 < 0.
        \end{cases}
  \end{aligned}
\end{equation*}
\end{lemma}
\begin{proof}
  For $1\le i < N/2$, $m+1 \le j \le 2N$ with \(m=\max\{2i, 3\}\),
  using \eqref{def:Tij}, \eqref{def:Vij}, \eqref{def:Kyx}, \Cref{lmm:Dyjleh2ya/2m2/r,lmm:Dh2Kyxi}, we have
  \begin{equation*}
    \begin{aligned}
      | V_{ij}| 
      & =  \left| \int_{x_{j-1}}^{x_{j}}\left( u(y) - \Pi_hu(y)\right) D_h^2 K_y (x_i) dy \right|  
       \le C h^2 \int_{x_{j-1}}^{x_{j}} \delta(y)^{\alpha/2-2/r} |x_i-y|^{-1-\alpha} dy.
    \end{aligned}
  \end{equation*}
  Since $y\ge x_{j-1} \ge x_{2i}$, $ y - x_i \simeq y$, and \(x_{i} \simeq x_{2i} \), it yields
  \begin{equation*}
    \begin{aligned}
      \sum_{j=m+1}^{N} |V_{ij}|
        & \le C h^2 \int_{x_{2i}}^{x_{N}} y^{-\alpha/2-2/r-1} dy                     
         \le C h^2 x_i^{-\alpha/2-2/r}.
    \end{aligned}
  \end{equation*}
  On the other hand, since $y-x_i \simeq T$ if $y\ge x_N = T$, there exist
  \begin{equation*}
    \begin{aligned}
      \sum_{j=N+1}^{2N-1} |V_{ij}|
       & \le C T^{-1-\alpha} h^2 \int_{x_{N}}^{x_{2N-1}} (2T-y)^{\alpha/2-2/r}  dy                       \\
       & \le \begin{cases}
             \frac{C}{\alpha/2-2/r+1}T^{-\alpha/2-2/r} \; h^2,                & \alpha/2-2/r+1 > 0 ,\\
             CrT^{-1-\alpha} h^2 \ln(N),    & \alpha/2-2/r+1 = 0 ,\\
             \frac{C}{|\alpha/2-2/r+1|} T^{-\alpha/2-2/r} \; h^{r\alpha/2+r}, & \alpha/2-2/r+1 < 0  .
           \end{cases}
    \end{aligned}
  \end{equation*}
  Finally, by \Cref{lmm:Dyj1}, one has
  \begin{equation*}
    |V_{i,2N}| \le C T^{-1-\alpha} h_{2N}^{\alpha/2+1} = C T^{-\alpha/2} h^{r\alpha/2+r}.
  \end{equation*}
  Then, the desired result in Case 1 is obtained.
  We can similarly prove for Case 2, the details are omitted here.
\end{proof}

Immediately, we can calculate $R_1, R_2$ from \eqref{eq:departR1R2}.
\begin{lemma} \label{lmm:R1R2}
For \(i=1, 2\), we have
  \begin{equation*}
    \begin{aligned}
      |R_i| \le C h^2 x_i^{-\alpha/2-2/r} +
      \begin{cases}
        C h^2,             & \alpha/2-2/r+1 > 0 ,\\
        C h^2 \ln(N) ,     & \alpha/2-2/r+1 = 0 ,\\
        C h^{r\alpha/2+r}, & \alpha/2-2/r+1 < 0.
      \end{cases}
    \end{aligned}
  \end{equation*}
\end{lemma}
\begin{proof}
  According to \eqref{eq:departR1R2}, \Cref{lmm:sumSij13,lmm:Ri-I5-1}, the desired result is obtained.
\end{proof}

For \(R_i\) with \(3\le i\le N\), the terms \(\{I_1, I_2, I_3, I_4\}\) in \eqref{eq:depart-Ri} remain to be estimated.
  \begin{lemma} \label{lmm:Ri-I1}
    Let \(I_1 = \sum_{j=1}^{k-1} V_{ij}\) be defined by \eqref{eq:depart-Ri}.
    Then we have, 
    for \(3\le i \le N, k=\lceil\frac{i}{2}\rceil\),
    \begin{equation*}
      \sum_{j=1}^{k-1} |V_{ij}| \le \begin{cases}
        C h^2 x_i^{-\alpha/2-2/r} ,        & \alpha/2-2/r+1 > 0, \\
        C h^2 x_i^{-1-\alpha} \ln(i),      & \alpha/2-2/r+1 = 0 ,\\
        C h^{r\alpha/2+r} x_i^{-1-\alpha}, & \alpha/2-2/r+1 < 0.
      \end{cases}
    \end{equation*}
  \end{lemma}
\begin{proof}
  According to \eqref{def:Vij}, \Cref{lmm:Dyj1,lmm:Dh2Kyxi}, it yields
  \begin{equation*}
    |V_{i1}| \le C \int_{0}^{x_1} x_1^{\alpha/2} |x_i-y|^{-1-\alpha}dy \simeq x_1^{\alpha/2+1} x_i^{-1-\alpha} = T^{\alpha/2+1} h^{r\alpha/2+r} x_i^{-1-\alpha}.
  \end{equation*}
  Using \Cref{lmm:Dyjleh2ya/2m2/r}, \Cref{lmm:Dh2Kyxi} and $y\le x_{k-1} < 2^{-r}x_i $, \(x_i - y \simeq x_i\), we have
  \begin{equation*}
    \begin{aligned}
      |V_{ij}| 
             \le C h^2 \int_{x_{j-1}}^{x_{j}} y^{\alpha/2-2/r} x_i^{-1-\alpha} dy, \quad 2 \le j \le k-1,
    \end{aligned}
  \end{equation*}
  and
  \begin{equation*}
    \begin{aligned}
      \sum_{j=2}^{k-1} |V_{ij}|
        \le C h^{r\alpha/2+r} x_i^{-1-\alpha} + C h^2 x_i^{-1-\alpha} \int_{x_1}^{x_{\lceil\frac{i}{2}\rceil-1}} y^{\alpha/2-2/r} dy.
    \end{aligned}
  \end{equation*}
  Moreover we can check that
  \begin{equation*}
    \begin{aligned}
      \int_{x_1}^{x_{\lceil\frac{i}{2}\rceil-1}} y^{\alpha/2-2/r} dy
       & \le \begin{cases}
               \frac{1}{\alpha/2-2/r+1} (2^{-r} x_i)^{\alpha/2-2/r+1}  ,& \alpha/2-2/r+1 > 0, \\
               \ln(2^{-r} x_i) - \ln(x_1)                              ,& \alpha/2-2/r+1 = 0, \\
               \frac{1}{|\alpha/2-2/r+1|} x_1^{\alpha/2-2/r+1}         ,& \alpha/2-2/r+1 < 0.
             \end{cases}
    \end{aligned}
  \end{equation*}
  The proof is completed.
\end{proof}

Subsequently, we turn our attention to $I_3=\sum_{j=k+1}^{m-1} S_{ij}$ 
with $m=2i$ for $3\le i<N/2$ and $m=2N-\lceil N/2 \rceil +1$ for $N/2 \le i \le N$ in \eqref{def:m2c}.
  \begin{lemma} \label{lmm:estimate-Sij}
  Let \(I_3 = \sum_{j=k+1}^{m-1} S_{ij}\) be defined by \eqref{eq:depart-Ri}. Then we have
    \paragraph{Case 1}
    For \(N/2 \le i \le N\), $m=2N-\lceil N/2 \rceil +1$, there exist
    \begin{equation*}
      |S_{ij}| \le C (h^3 + (r-1)h^2) (T-x_i+h_N)^{1-\alpha}, \quad j=N, N+1,
    \end{equation*}
    and
    \begin{equation*}
      \sum_{j=N+2}^{m-1} \left|S_{ij} \right| 
      \le C h^2 + C (r-1) h^2 (T-x_i + h_N)^{1-\alpha}.
    \end{equation*}
    \paragraph{Case 2}
    For \(3\le i \le N-1\), $k=\lceil \frac{i}{2} \rceil$,  
    there exist
    \begin{equation*}
       \sum_{j=k+1}^{\min\{m-1, N-1\}} \left|S_{ij} \right| \le C h^2 x_i^{-\alpha/2-2/r}, \quad
      \sum_{j=\lceil\frac{N}{2}\rceil+1}^{N-1} \left|S_{Nj} \right| 
      \le C h^2 + C (r-1) h^2 h_N^{1-\alpha}.
    \end{equation*}
  \end{lemma}
\begin{proof}

  Case 1: From \eqref{eq:Sij-int}, using \(\theta(1-\theta) h_j \le |y_j^\theta - x_i|\), \Cref{lmm:hilexi,lmm:d2Pj-itle,lmm:dQj-itle}, it yields
  \begin{equation*}
    \begin{aligned}
      |S_{ij}| &\le C (h_j^3 + (r-1)h_j^2) \int_{0}^1 |y_j^\theta - x_i|^{1-\alpha} d\theta , \quad j=N, N+1
    \end{aligned}
  \end{equation*}
  with 
      $\int_{0}^{1} |y_j^\theta - x_i|^{1-\alpha} dy 
      \simeq (|x_j - x_i| + h_N)^{1-\alpha} $.

  On the other hand, for $j\ge N+2$, $x_i \simeq x_j \simeq T$, we have
  \begin{equation*}
    \begin{aligned}
      |S_{ij}| &\le C h_j^2 \int_{0}^1 \left(|y_j^\theta - x_i|^{1-\alpha}+ (r-1)|y_j^\theta - x_i|^{-\alpha}\right) h_j d\theta  \\
      &\le C h^2 \int_{x_{j-1}}^{x_{j}} |y - x_i|^{1-\alpha} + (r-1)|y - x_i|^{-\alpha} dy .
    \end{aligned}
  \end{equation*}
  It implies that
  \begin{equation*}
    \begin{aligned}
      \sum_{j=N+2}^{2N-\lceil\frac{N}{2}\rceil} |S_{ij}| 
      &= C h^2 \int_{x_{N+1}}^{x_{2N-\lceil\frac{N}{2}\rceil}} |y - x_i|^{1-\alpha} + (r-1)|y - x_i|^{-\alpha} dy \\
      & \le C h^2 \left( T^{2-\alpha} + (r-1) (T-x_i + h_N)^{1-\alpha} \right)  .
    \end{aligned}
  \end{equation*}
  
Case 2: for $3\le i \le N-1$, $k+1\le j\le \min\{m-1, N-1\}$, using \Cref{lmm:hilexi,lmm:d2Pj-itle,lmm:dQj-itle}, $x_i \simeq x_j$ and $h_i \simeq h_j$, we have
    \begin{equation*} 
      \begin{aligned}
        |S_{ij}| & \le C h_j^2 x_i^{\alpha/2-4} \int_{0}^{1}  |y_j^\theta - x_i|^{1-\alpha}  h_j d\theta 
         =  C h^2 x_i^{\alpha/2-2-2/r} \int_{x_{j-1}}^{x_{j}}  |y - x_i|^{1-\alpha}  dy,
      \end{aligned}
    \end{equation*}
  \begin{equation*}
    \begin{aligned}
      \sum_{k+1}^{\min\{2i-1, N-1\}} |S_{ij}|
      &\le C h^2 x_i^{\alpha/2-2-2/r} \int_{x_{k}}^{x_{\min\{2i-1, N-1\}}} |y - x_i|^{1-\alpha}  dy  
        \le  C h^2 x_i^{-\alpha/2-2/r} .
    \end{aligned}
  \end{equation*}
  We can similarly prove the last inequality by Case 1.
  The proof is completed.
\end{proof}


Finally, we focus our error analysis on the terms \(I_2\) and \(I_4\). 
\begin{lemma} \label{lmm:Ri-I2-I4-ilN/2}
Let \(I_2, I_4\) be defined by \eqref{eq:depart-Ri}. Then we have
\paragraph{Case 1}
For \(3\le i \le N\), \(k=\lceil\frac{i}{2}\rceil\), there exists
\begin{equation*}
  I_2 = \frac{2}{h_i + h_{i+1}}
  \left( \frac{1}{h_{i+1}} (T_{i+1, k} +  T_{i+1, k+1})
  - \left(\frac{1}{h_{i}}+\frac{1}{h_{i+1}}\right) T_{i,k} \right) \le C h^2 x_i^{-\alpha/2-2/r} .
\end{equation*}
\paragraph{Case 2}
For \(3\le i < N/2\), $m=2i$, there exists
\begin{equation*}
  I_4 = \frac{2}{h_i + h_{i+1}}
  \left( \frac{1}{h_{i}} (T_{i-1, 2i} +  T_{i-1, 2i-1})
  - \left(\frac{1}{h_{i}}+\frac{1}{h_{i+1}}\right) T_{i,2i} \right) \le C h^2 x_i^{-\alpha/2-2/r} .
\end{equation*}
\paragraph{Case 3}
For \(N/2 \le i \le N\), \(m=N-\lceil\frac{N}{2}\rceil+1\), there exists
\begin{equation*}
  I_4 = \frac{2}{h_i + h_{i+1}}
  \left( \frac{1}{h_{i}} (T_{i-1, m} +  T_{i-1, m-1})
  - \left(\frac{1}{h_{i}}+\frac{1}{h_{i+1}}\right) T_{i,m} \right) \le C h^2 .
\end{equation*}
\end{lemma}
\begin{proof}
Since
\begin{equation} \label{eq:I24-depart}
  \begin{aligned}
     & \frac{1}{h_{i+1}} \left(T_{i+1, k} +  T_{i+1, k+1}\right)
    - \left(\frac{1}{h_{i}}+\frac{1}{h_{i+1}}\right) T_{i,k}   \\
     & = \frac{1}{h_{i+1}} \left(T_{i+1, k} -  T_{i, k}\right) + \frac{1}{h_{i+1}} \left(T_{i+1, k+1} -  T_{i, k}\right) + \left(\frac{1}{h_{i+1}} - \frac{1}{h_{i}}\right) T_{i,k} .
  \end{aligned}
\end{equation}
According to \(x_i - x_k \simeq x_i \simeq x_k\), \Cref{lmm:Dyjleh2ya/2m2/r,lmm:Dh2Kyxi,lmm:hilexi}, we have
\begin{equation*}
  \begin{aligned}
    \frac{1}{h_{i+1}} (T_{i+1, k} -  T_{i, k})
     & = \int_{x_{k-1}}^{x_k} (u(y)-\Pi_hu(y)) D_h K_y (x_i) dy 
     \le C  h^2 x_i^{-\alpha/2-2/r} h_k .
  \end{aligned}
\end{equation*}

From \Cref{lmm:Dyj,lmm:Tij-normalized} and \eqref{def:Qj-itheta-jlN}, 
we can obtain
\begin{equation*}
  \begin{aligned}
    \frac{1}{h_{i+1}} \left(T_{i+1, k+1} -  T_{i, k}\right) 
     & = \int_{0}^{1} \frac{\theta(\theta-1)}{2} \frac{Q_{i,k;3}^\theta(x_{i+1})u''(\eta_{k+1}^\theta) - Q_{i,k;3}^\theta(x_i) u''(\eta_{k}^\theta)}{h_{i+1}} d\theta 
  \end{aligned}
\end{equation*}
with \(\eta_{k}^\theta \in (x_{k-1}, x_k)\) and \(\eta_{k+1}^\theta \in (x_k, x_{k+1})\).
Using \Cref{lmm:dQj-itle,lmm:hilexi}, we have
\begin{equation*}
  \frac{1}{h_{i+1}} |T_{i+1, k+1} -  T_{i, k}| \le C h^2 x_i^{-\alpha/2-2/r} h_k.
\end{equation*}

For the third term in \eqref{eq:I24-depart}, using $h_i \simeq h_k$, \Cref{lmm:hilexi,lmm:hi1-hi,lmm:Dyjleh2ya/2m2/r}, it yields
\begin{equation*}
  \begin{aligned}
    \frac{h_{i+1}-h_{i}}{h_i h_{i+1}} T_{i,k}
      \le C(r-1) h^2 x_i^{-\alpha/2-2/r} h_k.
  \end{aligned}
\end{equation*}
Then, the desired result of Case 1 is obtained.
The Case 2 and Case 3 for \(I_4\) can be similarly proven as the way in Case 1; the details are omitted here.
\end{proof}


\begin{proof} [\bf Proof of \Cref{thm:Ri-ilessN/2}]
For \(1\le i < N/2\) 
with $m=2i$ in \eqref{eq:depart-Ri}, 
combining
\Cref{lmm:R1R2}, 
\Cref{lmm:Ri-I1}, 
Cases 1 and 2 of \Cref{lmm:Ri-I2-I4-ilN/2}, 
Case 2 of \Cref{lmm:estimate-Sij} and
Case 1 of \Cref{lmm:Ri-I5-1}, 
the proof is completed.
\end{proof}

\begin{proof} [\bf Proof of \Cref{thm:Ri-N/2le-i-leN}]
For \(N/2 \le i \le N \) 
with $m=2N-\lceil N/2 \rceil +1$ in \eqref{eq:depart-Ri},
we split \(I_3\) as
 \begin{equation} \label{eq:I3-depart}
  \begin{aligned}
    I_3 &= \sum_{j=k+1}^{m-1} S_{ij}
         = \sum_{j=k+1}^{N-1} S_{ij} + (S_{iN}+S_{i,N+1}) + \sum_{j=N+2}^{m-1}  S_{ij}.
  \end{aligned}
\end{equation}
According to
\Cref{lmm:Ri-I1}, 
Cases 1 and 3 of \Cref{lmm:Ri-I2-I4-ilN/2}, 
\Cref{lmm:estimate-Sij}
and 
Case 2 of \Cref{lmm:Ri-I5-1}, 
the desired result is obtained.
\end{proof}

\section{Convergence analysis}
\label{sec:proof-convergence}

We can now prove our main convergence result for \Cref{thm:convergence}.

\subsection{Some properties of the stiffness matrix}

In this subsection, we show some properties of the stiffness matrix \(A\) defined by \eqref{eq:equation_matrix} and construct an appropriate  right-preconditioner for the resulting matrix algebraic equation.

\begin{lemma} \label{lmm:AisM}
  The stiffness matrix \(A\) defined by \eqref{eq:equation_matrix} is strictly diagonally dominant, with positive entries on the main diagonal and negative off-diagonal entries.
  In particular, there exists a constant \(C_A\) such that
  \begin{equation*}
    \begin{aligned}
      \sum_{j=1}^{2N-1} a_{ij}
      \ge  C_A (x_i^{-\alpha} + (2T-x_i)^{-\alpha}) \quad \text{with} \quad C_A= \frac{\kappa_\alpha(\alpha-1)}{\Gamma(2-\alpha)} 2^{-r\alpha}.
    \end{aligned}
  \end{equation*}
\end{lemma}
\begin{proof}
From \eqref{eq:aij}, there exist
 \begin{equation*}
     a_{ii} = \frac{\kappa_\alpha}{\Gamma(4-\alpha)}\frac{4}{h_i h_{i+1}} \left( h_i^{2-\alpha} + h_{i+1}^{2-\alpha} - (h_i+h_{i+1})^{2-\alpha} \right) > 0,
 \end{equation*}
 where we use $ 1 + t^\theta >  (1+t)^\theta$ with $t=\frac{h_{i+1}}{h_{i}}$ for $\theta\in (0,1)$.
 
 Let \(s = \frac{h_{i-1}}{h_{i}}\) and \(t=\frac{h_{i+1}}{h_{i}}\), 
 by \Cref{eq:ineq_fxy},  we can check that
 \begin{equation*}
     \begin{aligned}
        a_{i,i-1}
         &=\frac{-\kappa_\alpha}{\Gamma(4-\alpha)} \frac{1}{h_{i-1} h_{i} h_{i+1}} \Big(
            h_{i+1} h_{i-1}^{3-\alpha} - (h_{i}+h_{i+1})(h_{i-1}+h_{i})^{3-\alpha}     \\
            & \quad + h_i(h_{i-1} + h_{i}  + h_{i-1})^{3-\alpha} + (h_{i-1}+h_{i}) (h_{i}+h_{i+1})h_i^{2-\alpha}  \\
            &\quad - (h_{i-1}+h_{i})(h_{i} + h_{i+1})^{3-\alpha}  + h_{i-1} h_{i+1} h_{i}^{2-\alpha} + h_{i-1} h_{i+1}^{3-\alpha}
         \Big) < 0.
     \end{aligned}
 \end{equation*}
 
 For $|i-j|\ge 2$, $x_{i+1} - y$, $x_{i} - y$ and $x_{i-1} - y$ have the same sign ($> 0$ or $< 0$) for $y\in (x_{j-1}, x_{j+1})$, it yields
    $ \frac{h_i}{h_{i}+h_{i+1}} |x_{i+1} - y| +  \frac{h_{i+1}}{h_{i}+h_{i+1}} |x_{i-1} - y| = |x_i - y|$.
 Since $x^{1-\alpha}$ is a convex function for $\alpha\in(1,2)$, by Jensen's inequality, we have
 \begin{equation*}
     \frac{h_i}{h_{i}+h_{i+1}} |x_{i+1} - y|^{1-\alpha} +  \frac{h_{i+1}}{h_{i}+h_{i+1}} |x_{i-1} - y|^{1-\alpha} > |x_i - y|^{1-\alpha},
 \end{equation*}
 which implies that \(D_h^2 K_y(x_i) > 0\) 
 by \eqref{def:Dh2} and \eqref{def:Kyx}.
 Thus, from \eqref{eq:aij}, we get
 \begin{equation*}
     a_{ij} = - D_h^2 I^{2-\alpha} \phi_j (x_i)
      = - \int_{x_{j-1}}^{x_{j+1}} \phi_j(y) D_h^2 K_y(x_i) dy < 0 .
 \end{equation*}

 We next prove that the stiffness matrix \(A\) defined by \eqref{eq:equation_matrix} is strictly diagonally dominant.
 For the quadrature coefficients $\tilde{a}_{ij}$ in \eqref{eq:aij}, we calculate that
  \begin{equation*}
    \begin{aligned}
      \sum_{j=1}^{2N-1} \tilde{a}_{ij}
      = g_0(x_i) + g_{2N}(x_i) 
    \end{aligned}
  \end{equation*}
  with
  $  g_{0}(x) = \frac{-\kappa_\alpha}{\Gamma(4-\alpha)} \frac{|x-x_0|^{3-\alpha} - |x-x_1|^{3-\alpha}}{h_1} $ and
  $  g_{2N}(x) = \frac{-\kappa_\alpha}{\Gamma(4-\alpha)} \frac{|x_{2N}-x|^{3-\alpha} - |x_{2N-1}-x|^{3-\alpha}}{h_{2N}} $.
  It implies that
    $  \sum_{j=1}^{2N-1}a_{ij}        
      =     D_h^2 g_0(x_i) + D_h^2 g_{2N}(x_i) $ by \eqref{eq:aij}.
  
  For \(i=1\), there exists
  \begin{equation*}
    \begin{aligned}
      D_h^2 g_0(x_1) 
        &= 2\kappa_\alpha \frac{1+(2^r-1)^{3-\alpha} + 2(2^r-1) - (2^r)^{3-\alpha} }{\Gamma(4-\alpha)2^r (2^r-1)} x_1^{-\alpha} 
        \! \ge  \frac{\kappa_\alpha(\alpha-1)}{\Gamma(2-\alpha)}2^{-r\alpha} x_1^{-\alpha} ,
    \end{aligned}
  \end{equation*}
  since
  $ h(t) = 2\left( 1+(t-1)^{3-\alpha} + 2(t-1) - t^{3-\alpha} \right) - (3-\alpha)(2-\alpha)(\alpha-1)( t^{2-\alpha} - t^{1-\alpha}) $ is a increasing function for $t=2^r\ge 1$ and $h(1)=0$.

  For \(i \ge 2\), by \Cref{lmm:Dh2simd2}, there exist $\xi \in (x_{i-1}, x_{i+1})$ and $\eta \in (x_0, x_1)$ such that
  \begin{equation*}
    \begin{aligned}
      D_h^2 g_0(x_i) &=  g_0''(\xi) 
                 =  -\kappa_\alpha \frac{|\xi-x_0|^{1-\alpha} - |\xi-x_1|^{1-\alpha}}{\Gamma(2-\alpha)h_1}                                               
                  = \frac{\kappa_\alpha(\alpha-1)}{\Gamma(2-\alpha)}  |\xi-\eta|^{-\alpha}      \\
                 & \ge \frac{\kappa_\alpha(\alpha-1)}{\Gamma(2-\alpha)} x_{i+1}^{-\alpha}  
                 \ge \frac{\kappa_\alpha(\alpha-1)}{\Gamma(2-\alpha)} 2^{-r\alpha} x_{i}^{-\alpha} .
    \end{aligned}
  \end{equation*}
  Then we have
  \(
    D_h^2 g_0(x_i) \ge C_A x_i^{-\alpha}
  \) with $C_A = \frac{\kappa_\alpha (\alpha-1)}{\Gamma(2-\alpha)}2^{-r\alpha} $ for $i\ge 1$.
  We can similarly prove
  \(
    D_h^2 g_{2N}(x_i) \ge C_A (2T-x_i)^{-\alpha} 
  \). The proof is completed.
\end{proof}

Let us first introduce the quasi-preconditioner
\begin{equation} \label{def:G}
  G = \text{diag}(\delta(x_1), ..., \delta(x_{2N-1})),
\end{equation}
where $\delta(x)$ is defined by \eqref{def:delta}.
Then we have
\begin{lemma} \label{lmm:AGhasSingularity}
  Let \(\tilde{B}:= AG\) and $A$ be defined by \eqref{eq:equation_matrix}. 
  Then the matrix $\tilde{B} = (\tilde{b}_{ij}) \in \mathbb{R}^{(2N-1) \times (2N-1)}$ has positive entries on the main diagonal and negative off-diagonal entries.
  In particular, there exist constants \(C_{\tilde{B}}, C_B\) such that
  \begin{equation*}
    \sum_{j=1}^{2N-1} \tilde{b}_{ij}
    \ge  C_B (T - \delta(x_{i}) + h_N)^{1-\alpha} -C_{\tilde{B}} (x_i^{1-\alpha} + (2T-x_i)^{1-\alpha}) .
  \end{equation*}
  with
  \(
    C_B = \frac{2\kappa_\alpha}{\Gamma(2-\alpha)}, \;
    C_{\tilde{B}} = \frac{\kappa_\alpha}{\Gamma(2-\alpha)} 2^{r(\alpha-1)} \). 
\end{lemma}
\begin{proof}
  From \eqref{eq:aij} and \eqref{def:G}, it yields
  \begin{equation*}
    \tilde{b}_{ij} = a_{ij} \delta(x_j) = -  \frac{2}{h_i + h_{i+1}} \left( \frac{1}{h_{i+1}} \tilde{a}_{i+1,j} - (\frac{1}{h_{i}}+\frac{1}{h_{i+1}})\tilde{a}_{i,j} + \frac{1}{h_{i}} \tilde{a}_{i-1,j}\right) \delta(x_j).
  \end{equation*}

  Since
  \(
    \delta(x) \equiv \Pi_h \delta(x) = \sum_{j=1}^{2N-1} \phi_j(x) \delta(x_j)
  \) by \eqref{def:interp} and \eqref{def:delta},
  from the definition of the quadrature coefficients $\tilde{a}_{ij}$ in \eqref{eq:aij}, we have
  \begin{equation*}
    \begin{aligned}
      \sum_{j=1}^{2N-1} \tilde{a}_{ij} \delta(x_j)
                  & = \sum_{j=1}^{2N-1} I^{2-\alpha} \phi_j(x_i) \delta(x_j) 
                    = I^{2-\alpha} \delta(x_i)
                  =  -p(x_i) + q(x_i)
    \end{aligned}
  \end{equation*}
  with
    $  p(x) = \frac{2 \kappa_\alpha}{\Gamma(4-\alpha)}|T-x|^{3-\alpha}$ and
    $  q(x) = \frac{\kappa_\alpha}{\Gamma(4-\alpha)}\left( x^{3-\alpha} + (2T-x)^{3-\alpha} \right) $.
  It implies that
      $\sum_{j=1}^{2N-1} \tilde{b}_{ij} =\sum_{j=1}^{2N-1} a_{ij} \delta(x_j)     
       =  D_h^2  p (x_i) -  D_h^2 q(x_i)$ .
  
  For \(i \neq N\), by \Cref{lmm:Dh2simd2}, there exists $\xi \in (x_{i-1}, x_{i+1})$ such that
  \begin{equation*}
    \begin{aligned}
      D_h^2 p(x_i) 
            & = \frac{2 \kappa_\alpha}{\Gamma(2-\alpha)} |T - \xi|^{1-\alpha}
            \ge \frac{2 \kappa_\alpha}{\Gamma(2-\alpha)} (T - \delta(x_{i}) + h_N)^{1-\alpha} 
    \end{aligned}
  \end{equation*}
  and for $i=N$, it yields
  \begin{equation*}
    \begin{aligned}
      D_h^2 p(x_N) 
          & = \frac{4 \kappa_\alpha}{\Gamma(4-\alpha) h_N^2} h_N^{3-\alpha}
          \ge C_B (T - \delta(x_N)+h_N)^{1-\alpha} .
    \end{aligned}
  \end{equation*}


  We can similarly prove the following inequality
  \begin{equation*}
    \begin{aligned}
      D_h^2 q(x_i) 
          & \le C_{\tilde{B}}  (x_{i}^{1-\alpha} + (2T-x_{i})^{1-\alpha}), \quad i=1,\cdots, 2N-1 .
    \end{aligned}
  \end{equation*}
  The proof is completed.
\end{proof}

Noted that $\tilde{B}=AG$ in \Cref{lmm:AGhasSingularity} is not diagonally dominant, e.g., 
$\sum_{j=1}^{2N-1} \tilde{b}_{ij} < 0 $ when $x_i$ is near the boundary.
We introduce the preconditioner $\lambda I + \mu G$ as following.
\begin{lemma}\label{thm:ALGisM}
  Let \(B := A(\lambda I+\mu G)\) 
    with $\lambda= 1+ 2^{r(\alpha-1)} T $, $\mu = (\alpha-1)2^{-r\alpha-1}$. 
  Then $B = (b_{ij})\in \mathbb{R}^{(2N-1) \times (2N-1)}$ is strictly diagonally dominant, with positive entries on the main diagonal and negative off-diagonal entries. Moreover, we have
  \begin{equation*}
    \sum_{j=1}^{2N-1} b_{ij} \ge C_A \left( (x_i^{-\alpha} + (2T-x_i)^{-\alpha}) + (T - \delta(x_{i}) + h_N)^{1-\alpha} \right).
  \end{equation*}
\end{lemma}
\begin{proof}
  From \Cref{lmm:AisM,lmm:AGhasSingularity}, we have 
  \begin{equation*}
      \begin{aligned}
          \sum_{j=1}^{2N-1} b_{ij} &= \sum_{j=1}^{2N-1} \left(\lambda a_{ij} + \mu \tilde{b}_{ij} \right) \\
           & \ge \lambda C_A \left( x_i^{-\alpha} + (2T-x_i)^{-\alpha} \right)
            - \mu C_{\tilde{B}} 2T \left( x_i^{-\alpha} + (2T-x_i)^{-\alpha} \right) \\
           & \;\; + \mu C_B \left( T-\delta(x_i) + h_N \right)^{1-\alpha},
      \end{aligned}
  \end{equation*}
  with \(\lambda = 1+2T C_{\tilde{B}} / C_B = 1+ 2^{r(\alpha-1)} T\) and \(\mu= C_A/C_B = (\alpha-1)2^{-r\alpha-1} \).
  The proof is completed.
\end{proof}

\subsection{Proof of \Cref{thm:convergence}}
\label{subsec:proof-convergence}
Let \(\epsilon_i = u(x_i) - u_i\) with $\epsilon_0 = \epsilon_{2N} = 0$.
Subtracting \eqref{eq:AU} from \eqref{eq:AUhat}, we get 
\begin{equation} \label{eq:Ae=t}
    A \epsilon = \tau,
\end{equation}
where $\epsilon = [\epsilon_1, \epsilon_2, ... , \epsilon_{2N-1}]^T$ and $\tau=[\tau_1, \tau_2, ... , \tau_{2N-1}]^T$ with $\tau_i$ in \eqref{def:truncation_error}.

Let $\lambda I + \mu G$ be the right-preconditioner and $B=A(\lambda I + \mu G)$ defined in \Cref{thm:ALGisM}.
Then we can rewrite \eqref{eq:Ae=t} as 
\begin{equation*}
    B (\lambda I+ \mu G)^{-1} \epsilon = \tau,
    \quad \text{i.e.} \quad
  \sum_{j=1}^{2N-1} b_{ij} \frac{\epsilon_j}{\lambda + \mu \delta(x_j)} = \tau_i .
\end{equation*}
Assume that 
 $ \left|\frac{\epsilon_{i_0}}{\lambda + \mu \delta(x_{i_0})} \right| 
  = \max\limits_{1\le j\le 2N-1} \left|\frac{\epsilon_j}{\lambda + \mu \delta(x_j)} \right|$. 
From \Cref{thm:ALGisM} with $b_{ii} > 0$ and $b_{ij} < 0, i\neq j$, it yields
\begin{equation*}
  \begin{aligned}
    |\tau_{i_0}| &= \left|\sum_{j=1}^{2N-1} b_{i_0, j} \frac{\epsilon_j}{\lambda + \mu \delta(x_j)}\right|
    \ge b_{i_0, i_0} \left| \frac{\epsilon_{i_0}}{\lambda + \mu \delta(x_{i_0})} \right| - \sum_{j\ne i_0} |b_{i_0, j}| \left| \frac{\epsilon_j}{\lambda + \mu \delta(x_j)} \right| \\
    &\ge b_{i_0, i_0} \left| \frac{\epsilon_{i_0}}{\lambda + \mu \delta(x_{i_0})} \right| - \sum_{j\ne i_0} |b_{i_0, j}| \left|\frac{\epsilon_{i_0}}{\lambda + \mu \delta(x_{i_0})} \right|
    = \sum_{j=1}^{2N-1} b_{i_0, j} \left|\frac{\epsilon_{i_0}}{\lambda + \mu \delta(x_{i_0})}\right|  .
  \end{aligned}
\end{equation*}

According to the above inequality, \Cref{thm:truncation-error} and \Cref{thm:ALGisM}, we have
\begin{equation*}
  \left|\frac{\epsilon_{i}}{\lambda + \mu \delta(x_{i})}\right| \le \left|\frac{\epsilon_{i_0}}{\lambda + \mu \delta(x_{i_0})}\right| 
  \le \frac{|\tau_{i_0}|}{\sum_{j=1}^{2N-1} b_{i_0, j}}
  \le C h^{\min\{\frac{r\alpha}{2}, 2\}} + C (r-1) h^2 .
\end{equation*}
Since \(\lambda + \mu\delta(x_{i}) \le \lambda + \mu T \), we can derive
\begin{equation*}
  |\epsilon_i| \le C(\lambda + \mu T) h^{\min\{\frac{r\alpha}{2}, 2\}}
  \le C \left(1+ (2^{r(\alpha-1)} + (\alpha-1) 2^{-r\alpha-1}) T \right) h^{\min\{\frac{r\alpha}{2}, 2\}}.
\end{equation*}
The proof is completed.

\begin{remark}
    Let $B=A (\lambda I + \mu G)$ with $\mu=0$ in the proof of \Cref{thm:convergence}, which means that there is no preconditioning.
    Thus, from \Cref{thm:truncation-error} and \Cref{lmm:AisM},
    we can only prove
    $$ |\epsilon_i| \le C h^{\min\{\frac{r\alpha}{2}, 2\}} + C(r-1) h^{3-\alpha} \le C h^{\min\{\frac{r\alpha}{2}, 3-\alpha\}}, \quad 1<\alpha <2, $$
    which may suffer from a severe order reduction.
\end{remark}

From \Cref{lmm:regularity-u}, the source term  $|f(x)| \le C \delta(x)^{-\alpha/2}$ may potentially be singular.
If the bound of \Cref{lmm:regularity-u} is replaced by the more general weaker regularity bound, we have the following result.
\begin{theorem}[Global Error with singualr source term]\label{thm:Singualrconvergence}
Let $\alpha \in (1,2)$ and
 \begin{equation*}
  \begin{split}
    &|u^{(l)}(x)| \le C 
    \delta(x)^{\sigma-l}, l=0, 1, 2, 3, 4,    \\
    &|f^{(l)}(x)| \le C
    \delta(x)^{\sigma-\alpha-l}, l=0, 1, 2
  \end{split}
  \end{equation*}
  with  \(\sigma \in (0, \frac{\alpha}{2}]\).
Let \(u_i\) be the approximate solution of \(u(x_i )\) computed by the discretization scheme \eqref{def:discrete_equation}. Then,
\begin{equation*} 
  \max_{1\le i \le 2N-1} |u_i - u(x_i)| \le C h^{\min\{r\sigma, 2\}}
\end{equation*}
for some positive constant \(C=C(\Omega, \alpha, \sigma, r, f)\).  
\end{theorem}
\begin{proof}
Similar to the performer in \Cref{thm:truncation-error,thm:convergence},
  it is easy to check the local truncation error
    \begin{equation*} 
      \begin{aligned}
        |\tau_i| &= | -D_h^{\alpha} \Pi_hu(x_i) - f(x_i) | \\
        &\le  C  h^{\min\{r\sigma, 2\}} \delta(x_i)^{-\alpha}
        + C(r-1) h^2 (T-\delta(x_{i}) + h_N)^{1-\alpha},
      \end{aligned}
    \end{equation*}
 and the global error
     $ \max\limits_{1\le i \le 2N-1} |u_i - u(x_i)| \le C h^{\min\{r\sigma, 2\}}$.
The proof is completed. 
    \end{proof}
\section{Numerical experiments}
\label{sec:experiments}

We use the difference-quadrature scheme \eqref{eq:equation_matrix} to solve the fractional Laplacian \eqref{eq:equation}  with both regular and singular source terms .
\subsection{Regular source term}

If \(f\equiv 1\), the exact (Getoor) solution \cite{Getoor1961,HuangO:14,RosOtonSerra:14} of the problem \eqref{eq:equation} is
\begin{equation*}
    u(x) = \frac{2^{-\alpha} \Gamma(\frac{1}{2})}{\Gamma(1+\frac{\alpha}{2})\Gamma(\frac{1+\alpha}{2}) } \left[ x(1-x) \right]^{\frac{\alpha}{2}}, \quad x\in \Omega=(0,1).
\end{equation*}

In the numerical experiments of this example, we measure the numerical errors by using the maximum nodal error (i.e., the discrete \(L^\infty\) norm):
\begin{equation*}
    E^N := \max_{0\le i\le 2N} |u(x_i) - u_i|.
\end{equation*}
The rate of convergence of \(E^N\) is computed in the usual way, viz.,
\begin{equation*}
    Rate^N = \log_2 \left( \frac{E^{N/2}}{E^{N}} \right).
\end{equation*}
\Cref{example:1} shows that the difference-quadrature method \eqref{eq:equation_matrix} has convergence order
$\mathcal{O}(h^{\min\{\frac{r\alpha}{2}, 2\}})$, which agrees exactly with \Cref{thm:convergence}.

\begin{table}[h]\fontsize{9.5pt}{12pt}\selectfont
  \begin{center}
   \caption {Maximum nodal errors showing convergence rate $\mathcal{O}(h^{\frac{\alpha}{2}})$ with $r=1$ and $\mathcal{O}(h^{2})$ with $r=4/\alpha$.} \vspace{5pt}
 \begin{tabular*}{\linewidth}{@{\extracolsep{\fill}}*{10}{c}}                                    \hline
 \multicolumn{8}{c}{ $r=1$}~~~~~~~~~~~~~~~~~~~~~~~~~~~~~~~~~~~~~~~~~~~~~~~~$r=4/\alpha$  \\ \hline
 $N$  &$\alpha=1.2$ &  Rate     & $\alpha=1.8$  & Rate       & $\alpha=1.2$ &   Rate   & $\alpha=1.8$ &   Rate   \\\hline
 100 &  1.1269e-3   &           & 2.7320e-5    &            & 4.1583e-5   &         & 7.6424e-6  &  \\
 200 &  7.4281e-4  & 0.6013     & 1.4829e-5    & 0.8815     & 1.0628e-5   & 1.9682  & 2.0649e-6  & 1.8880  \\
 400 &  4.8986e-4  &  0.6006    & 7.9970e-6    & 0.8909    & 2.6919e-6   & 1.9811  & 5.5008e-7  & 1.9083  \\
 800 &  3.2311e-4  &  0.6003     & 4.2989e-6    & 0.8955     & 6.7824e-7   & 1.9888  & 1.4495e-7  & 1.9240  \\
 \hline
 \end{tabular*}\label{example:1}
 \end{center}
 \end{table}

\subsection{Singular source term}

We take the singular source term \(f(x)=x^{\sigma-\alpha}\), \(\sigma\in(0, \frac{\alpha}{2}]\) with $\sigma=0.4$ in \eqref{eq:equation}.
Since the analytic solution is unknown, the convergence rate of the numerical results is computed by
\begin{equation*}
  Rate^N = \log_2 \left( \frac{E^{N/2}}{E^{N}} \right)
  \quad \text{with} \quad
  E^N = \max_{0\le i\le N} |u^{N/2}_i - u^{N}_{2i}|.
\end{equation*}


\begin{table}[h]\fontsize{9.5pt}{12pt}\selectfont
  \begin{center}
   \caption {Maximum nodal errors showing convergence rate $\mathcal{O}(h^{\sigma})$ with $r=1$ and $\mathcal{O}(h^{2})$ with $r=2/\sigma$.} \vspace{5pt}
 \begin{tabular*}{\linewidth}{@{\extracolsep{\fill}}*{10}{c}}                                    \hline
 \multicolumn{8}{c}{ $r=1$}~~~~~~~~~~~~~~~~~~~~~~~~~~~~~~~~~~~~~~~~~~~~~~~~$r=2/\sigma$  \\ \hline
 $N$  &$\alpha=1.2$ &  Rate     & $\alpha=1.8$  & Rate       & $\alpha=1.2$ &   Rate   & $\alpha=1.8$ &   Rate   \\\hline
 100 &  2.9193e-2   &           & 5.6776e-2    &            & 2.7820e-3   &         & 5.0311e-3  &  \\
 200 &  2.2619e-2  & 0.3681     & 4.3468e-2    & 0.3853     & 6.9631e-4   & 1.9983  & 1.3190e-3  & 1.9315  \\
 400 &  1.7435e-2  &  0.3755    & 3.3112e-2    & 0.3926    & 1.7418e-4   & 1.9992  & 3.4157e-4  & 1.9492  \\
 800 &  1.3395e-2  &  0.3804     & 2.5161e-2    & 0.3962     & 4.3557e-5   & 1.9996  & 8.7691e-5  & 1.9617  \\
 \hline
 \end{tabular*}\label{example:2}
 \end{center}
 \end{table}

\Cref{example:2} shows that the difference-quadrature method \eqref{eq:equation_matrix} has convergence order
$\mathcal{O}(h^{\min \{r\sigma, 2\}})$, which agrees with \Cref{thm:Singualrconvergence}.

\section*{Acknowledgments}
This work was supported by the National Natural Science Foundation of China under Grant No. 12471381 and Science Fund for Distinguished Young Scholars of Gansu Province under Grant No. 23JRRA1020.

\appendix

\section{Approximations of difference and interpolation}
In this appendix, we provide some approximations for the second-order difference quotients $D_h^2$ and the interpolation error $u(x)-\Pi_h u(x)$.
\begin{lemma} \label{lmm:Dh2simd2}
  Let $D_h^2$ be the difference quotient operator defined by \eqref{def:Dh2}.
  If \(g(x)\in C(\bar{\Omega}) \cap C^2(\Omega)\),
  there exists \(\xi\in (x_{i-1}, x_{i+1})\), $i=1, 2, ..., 2N-1$, such that
  \begin{equation*} \label{eq:Dh2simd2}
    \begin{aligned}
      D_h^2 g(x_i) & = g''(\xi), \quad \xi \in (x_{i-1}, x_{i+1}).
    \end{aligned}
  \end{equation*}
  Moreover, if \(g(x) \in C(\bar{\Omega}) \cap C^4(\Omega)\), then we have
  \begin{equation*} \label{eq:Dh2simd4}
    \begin{aligned}
      & D_h^2 g(x_i) = g''(x_{i}) + \frac{h_{i+1}-h_{i}}{3} g'''(x_{i}) \\
      &  +  \frac{2}{h_i + h_{i+1}}\left( \frac{1}{h_{i}}\int_{x_{i-1}}^{x_i} g''''(y) \frac{(y-x_{i-1})^3}{3!} dy + \frac{1}{h_{i+1}} \int_{x_i}^{x_{i+1}} g''''(y) \frac{(x_{i+1}-y)^3}{3!} dy \right) .
    \end{aligned}
  \end{equation*}
\end{lemma}
\begin{proof}
  By Taylor series expansion, we obtain
  \begin{gather*}
    g(x_{i-1}) = g(x_{i}) - (x_{i}-x_{i-1}) g'(x_{i}) + \frac{(x_{i}-x_{i-1})^2}{2} g''(\xi_1), \quad \xi_1 \in (x_{i-1}, x_{i}),        \\
    g(x_{i+1}) = g(x_{i}) + (x_{i+1}-x_{i}) g'(x_{i}) + \frac{(x_{i+1}-x_{i})^2}{2} g''(\xi_2), \quad \xi_2 \in (x_{i}, x_{i+1}).
  \end{gather*}
  From \eqref{def:Dh2} and intermediate value theorem, it yields
  \begin{equation*}
    D_h^2 g(x_i) = \frac{h_i}{h_i + h_{i+1}} g''(\xi_1) + \frac{h_{i+1}}{h_i + h_{i+1}} g''(\xi_2) = g''(\xi) , \quad \xi \in [\xi_{1}, \xi_{2}].
  \end{equation*}
  The second equation can be derived in a similar manner. 
  The proof is completed.
\end{proof}

\begin{lemma} \label{lmm:Dyj}
  Let \(y_j^\theta = (1-\theta) x_{j-1} + \theta x_j, \theta\in (0,1)\) with $2\le j\le 2N-1$. Then,
  \begin{equation*}
    \begin{aligned}
      u(y_j^\theta) - \Pi_hu(y_j^\theta) & = -\frac{\theta (1-\theta)}{2} h_j^2 u''(\xi), \quad \xi \in (x_{j-1}, x_j),
    \end{aligned}
  \end{equation*}
  \begin{equation*}
    \begin{aligned}
      u(y_j^\theta) - \Pi_hu(y_j^\theta) = & -\frac{\theta (1-\theta)}{2} h_j^2 u''(y_j^\theta)
      + \frac{\theta (1-\theta)}{3!} h_j^3 \left(\theta^2 u'''(\eta_1) - (1-\theta)^2 u'''(\eta_2) \right)
    \end{aligned}
  \end{equation*}
  with \(\eta_1 \in (x_{j-1}, y_j^\theta), \eta_2 \in (y_j^\theta, x_j)\).
\end{lemma}
\begin{proof}
  Using Taylor series expansion, we get
  \begin{gather*}
    u(x_{j-1}) = u(y_j^\theta) - \theta h_{j} u'(y_j^\theta) + \frac{\theta^2 h_{j}^2}{2!} u''(\xi_1), \quad \xi_1 \in (x_{j-1}, y_j^\theta) ,\\
    u(x_{j}) = u(y_j^\theta) + (1-\theta) h_{j} u'(y_j^\theta) + \frac{(1-\theta)^2 h_{j}^2}{2!} u''(\xi_2) , \quad \xi_2 \in (y_j^\theta, x_j) ,
  \end{gather*}
  which implies that
  \begin{equation*}
    \begin{aligned}
      u(y_j^\theta) - \Pi_hu(y_j^\theta) 
      & = u(y_j^\theta) - (1-\theta) u(x_{j-1}) - \theta u(x_{j})      \\
      & = -\frac{\theta (1-\theta)}{2} h_j^2 ( \theta u''(\xi_1) + (1-\theta) u''(\xi_2) ) \\
      & = -\frac{\theta (1-\theta)}{2} h_j^2 u''(\xi), \quad \xi \in [\xi_1, \xi_2].
    \end{aligned}
  \end{equation*}
  The second equation can be derived in a similar manner. 
  The proof is completed.
\end{proof}

\begin{lemma} \label{lmm:Dyjleh2ya/2m2/r}
  For any $y\in (x_{j-1}, x_{j})$, \(2\le j \le 2N-1\), there exists 
  \begin{equation*}
    |u(y)-\Pi_hu(y)| \le h_j^2 \max_{\xi\in [x_{j-1}, x_j]}|u''(\xi)| \le C h^2 \delta(y)^{\alpha/2-2/r}.
  \end{equation*}
\end{lemma}
\begin{proof}
  According to \Cref{lmm:Dyj,lmm:regularity-u,lmm:hilexi}, the desired result is obtained.
\end{proof}

\begin{lemma} \label{lmm:Dyj1}
  For any \(x\in [x_{j-1}, x_j]\), $1\le j \le 2N$, there exists
  \begin{equation*}
    \begin{aligned}
      |u(x) - \Pi_hu(x)| & =\! \left| \frac{x_{j}-x}{h_j} \int_{x_{j-1}}^x u'(y) dy - \frac{x-x_{j-1}}{h_j} \int_{x}^{x_{j}} u'(y) dy \right| 
                       \le \!\int_{x_{j-1}}^{x_{j}} |u'(y)| dy.
    \end{aligned}
  \end{equation*}
  In particular, we have
    $|u(x) - \Pi_hu(x)| \le 
    C\frac{2}{\alpha} h_1^{\alpha/2}$ for $x\in (0, x_1) \cup (x_{2N-1}, 2T)$.
\end{lemma}
\begin{proof}
    From the definition of $\Pi_h u(x)$ in \eqref{def:interp} and using $u(x) = u(x_i) + \int_{x_i}^{x} u'(y) dy$, the proof is completed.
\end{proof}

\section{Bound estimates}
Set
\(
  h_i = x_{i} - x_{i-1}
\) for $j = 1, 2, ... ,  2N$ 
and define $h := \frac{1}{N}$. 
The following bounds are needed in several places.
\begin{lemma} \label{lmm:hi1-hi}
For \(i=1,2,\cdots,2N-1\), there exists a constant \(C\) such that 
\begin{equation*}
  |h_{i+1} - h_{i}| \le C(r-1) h^2 \delta(x_i)^{1-2/r} , \quad r\ge 1.
\end{equation*}
\end{lemma}
\begin{proof}
According to the definition of $h_i=x_i-x_{i-1}$ as defined in \eqref{def:xj}, we obtain
\begin{equation*}
  \begin{aligned}
    h_{i+1} - h_{i} =
    \begin{cases}
      T \left( \left(\frac{i+1}{N}\right)^r - 2\left(\frac{i}{N}\right)^r + \left(\frac{i-1}{N}\right)^r  \right) ,           & 1\le i\le N-1,    \\
      0,    & i=N,    \\
      -T \left( \left(\frac{2N-i-1}{N}\right)^r - 2\left(\frac{2N-i}{N}\right)^r + \left(\frac{2N-i+1}{N}\right)^r  \right) , & N+1\le i\le 2N-1 .    \\
    \end{cases}
  \end{aligned}
\end{equation*}
Since $(i+1)^r - 2i^r + (i-1)^r \simeq r(r-1)i^{r-2}$ for $i\ge 1$,
the desired result is obtained.
\end{proof}

\begin{lemma} \label{lmm:trucerr2d2f}
For $1\le i \le 2N-1$,
  there exists a constant \(C\) such that
    \begin{equation*}
      \begin{aligned}
        & \frac{2}{h_i + h_{i+1}} \left| \frac{1}{h_i} \int_{x_{i-1}}^{x_{i}} f''(y) \frac{(y-x_{i-1})^3}{3!} dy + \frac{1}{h_{i+1}} \int_{x_{i}}^{x_{i+1}} f''(y) \frac{(y-x_{i+1})^3}{3!} dy \right| \\
         & \quad \le C h^2 \delta(x_i)^{-\alpha/2-2/r} .
      \end{aligned}
    \end{equation*}
\end{lemma}
\begin{proof} \label{prf:trucerr2d2f}
  By \Cref{lmm:regularity-u}, for \(1 \le i \le 2N-1\), we have
  \begin{equation*}
    \begin{aligned}
      \left|\int_{x_{i-1}}^{x_{i}} f''(y)\frac{(y-x_{i-1})^3}{3!} dy \right|  \le C \int_{x_{i-1}}^{x_{i}} \delta(y)^{-\alpha/2-2} (y-x_{i-1})^3 dy .
    \end{aligned}
  \end{equation*}
  In particular, for \(i=1\), there exists
  \begin{equation*}
    \int_{x_{i-1}}^{x_{i}} \delta(y)^{-\alpha/2-2} (y-x_{i-1})^3 dy
    = \int_{0}^{x_{1}} y^{1-\alpha/2} dy
    = \frac{1}{2-\alpha/2} x_1^{2-\alpha/2} \simeq x_1^{-\alpha/2-2} h_1^4 .
  \end{equation*}
  For \(2\le i \le 2N-1\), by \Cref{lmm:hilexi}, it yields
  \begin{equation*}
    \begin{aligned}
      \int_{x_{i-1}}^{x_{i}} \delta(y)^{-\alpha/2-2} (y-x_{i-1})^3 dy 
        \simeq \int_{x_{i-1}}^{x_{i}} \delta(x_i)^{-\alpha/2-2} (y-x_{i-1})^3 dy
        \simeq \delta(x_i)^{-\alpha/2-2} h_i^4
    \end{aligned}
  \end{equation*}
  The desired result is obtained.
\end{proof}

\begin{lemma} \label{lmm:Dh2Kyxi}
  For all \(1 \le i \le 2N-1\), \(1\le j \le 2N\), $y\in (x_{j-1}, x_{j})$, there exist
  \begin{gather*}
    |D_h K_y (x_i)| \simeq | x_i - y|^{-\alpha} \quad\text{if}\quad [x_{j-1}, x_{j}] \cap [x_i, x_{i+1}] = \varnothing , \\
    D_h^2 K_y (x_i) \simeq | x_i - y |^{-1-\alpha} \quad\text{if}\quad [x_{j-1}, x_{j}] \cap [x_{i-1}, x_{i+1}] = \varnothing.
  \end{gather*}
\end{lemma}
\begin{proof}
  Since \(x_{i-1}-y\), \(x_{i}-y\) and \(x_{i+1}-y\) have the same sign, using \Cref{lmm:Dh2simd2} and $K_y(x) = \frac{\kappa_\alpha}{\Gamma(2-\alpha)}|x-y|^{1-\alpha}$ in \eqref{def:Kyx}, it yields
  \begin{equation*}
    \begin{aligned}
      |D_h K_y (x_i)| &= \frac{\kappa_\alpha(\alpha-1)}{\Gamma(2-\alpha)}|\xi-y|^{-\alpha}, \quad \xi\in (x_{i}, x_{i+1})  ,\\
      D_h^2 K_y (x_i) & = \frac{\kappa_\alpha \alpha(\alpha-1)}{\Gamma(2-\alpha)}|\xi-y|^{-1-\alpha}, \quad \xi\in (x_{i-1}, x_{i+1})  .
    \end{aligned}
  \end{equation*}
  Moreover, from \(|\xi-y| \simeq |x_i-y|\), the desired result is obtained.
\end{proof}

\begin{lemma} \label{lmm:sumSij13}
  There exists a constant $C$ such that
  \begin{gather*}
    \sum_{j=1}^{3} V_{1j} \le C h^2 x_1^{-\alpha/2-2/r}  \quad \text{and} \quad
    \sum_{j=1}^{4} V_{2j} \le C h^2 x_2^{-\alpha/2-2/r}  .
  \end{gather*}
\end{lemma}
\begin{proof}
  According \Cref{lmm:Dyj1}, \Cref{lmm:Dyjleh2ya/2m2/r}, \eqref{def:Tij}, \eqref{def:Vij}, it implies, 
  \begin{equation*}
    T_{ij} \le C x_1^{2-\alpha/2} \simeq h_1^2 \; h^2 x_1^{-\alpha/2-2/r} \simeq h_1^2 \; h^2 x_2^{-\alpha/2-2/r} \quad \text{for}\quad 0\le i \le 3, 1\le j \le 4.
  \end{equation*}
  The proof is completed.
\end{proof}

\begin{lemma} \label{ineq:a-b-theta}
Let $a,b> 0,\; \theta \in [0,1]$. Then we have
   $ b^{1-\theta}|a^{\theta}-b^{\theta}| \le |a-b|$. 
\end{lemma}
\begin{proof}
    Since $|t^\theta - 1| \le |t-1|$ with $t=\frac{a}{b} > 0$, the proof is completed.
\end{proof}

\begin{lemma} \label{eq:ineq_fxy}
    Let $x > 0, y\ge 1$ with $\alpha \in (1,2)$. Then we have
    \begin{equation*}
        \begin{aligned}
    f(x,y) =& xy(1+x^{2-\alpha} +y^{2-\alpha}) + (1+x+y)^{3-\alpha} \\
         &- (1+x)(1+y)\left( (1+x)^{2-\alpha} + (1+y)^{2-\alpha} - 1 \right) > 0.
        \end{aligned}
    \end{equation*}
\end{lemma}
\begin{proof}
    The first and second derivatives of $f(x,y)$ with respect to $x$ are
    \begin{gather*}
        \begin{aligned}
        \partial_x f(x,y) =&  
            (3-\alpha)\left[ x^{2-\alpha} y + (1+x+y)^{2-\alpha} - (1+x)^{2-\alpha}(1+y) \right] \\
            &+ 1+2y+y^{3-\alpha}-(1+y)^{3-\alpha} ,
        \end{aligned}  \\
        \partial_x^2 f(x, y) = (3-\alpha)(2-\alpha) \left( y x^{1-\alpha} + (1+x+y)^{1-\alpha} - (1+y) (1+x)^{1-\alpha} \right).
    \end{gather*}
    Since $\frac{y}{1+y}x + \frac{1}{1+y}(1+x+y) = 1+x$
    and $x^{1-\alpha}$ is convex for $x>0$,
    using Jensen's inequality, it yields
    $
        \frac{y}{1+y} x^{1-\alpha} + \frac{1}{1+y} (1+x+y)^{1-\alpha} > (1+x)^{1-\alpha} $,
    which implies $\partial_x^2 f(x, y) > 0$ and $\partial_x f(x, y)>\partial_x f(0, y)$. 
    Since $\partial_x f(0, y)>0$ for $y\ge 1$, we have $f(x,y)>f(0,y)=0$.
    The proof is completed.
\end{proof}

\section{Proofs for grid mapping functions}\label{sec:prfs-of-GMFs}
In this appendix, we provide the proofs of Lemmas \ref{lmm:gen-prop-of-MTFs}-\ref{lmm:dQj-itle} in \Cref{subsec:mesh-transport-functions}.
\begin{proof} [\bf Proof of \Cref{lmm:gen-prop-of-MTFs}]
  \label{prf:gen-prop-of-MTFs}
  The first two approximations can be derived from \eqref{def:xj} and \eqref{eq:prop-of-GMFs} with $2\le i, j \le 2N-2$.    \par
  We next prove $|y_{i,j}(\xi) - \xi| \simeq |x_j - x_i|$.
  From \eqref{def:yij}, we have $y_{i,j}(\xi)-\xi = 0$ if $i=j$. 
  Without loss of generality, if $i< j$, then $ y_{i,j}(\xi) - \xi \le x_{j+1} - x_{i-1} \simeq x_j - x_i$.
  Since the second derivatives of $|y_{i,j}(\xi) - \xi|$ is less than zero by \eqref{def:yij}, which implies it is concave. Thus, $|y_{i,j}(\xi) - \xi| \ge \min\{x_{j-1}-x_{i-1}, x_{j+1}-x_{i+1}\} \simeq |x_{j} - x_{i}| $.  \par
  From \eqref{def:yij}, \eqref{def:hij}, \eqref{eq:prop-of-GMFs}, using the approximation above, there exists
  $ h_{i,j}(\xi) = y_{i,j}(\xi) - y_{i,j-1}(\xi) = y_{j-1, j}(y_{i,j-1}(\xi)) - y_{i,j-1}(\xi) \simeq x_{j} - x_{j-1}=h_j. $
  
  The final estimate can be obtained since \(y_{i,j-1}(\xi) - \xi\), \(y_{i,j}(\xi) - \xi\) have the same sign \((\ge 0\) or \(\le 0)\).
\end{proof}

\begin{proof} [\bf Proof of \Cref{lmm:esitmate-of-MTFs-1}]
  From \eqref{def:hij} and \eqref{def:yij}, we can see that
  \begin{equation*}
    \begin{aligned}
      h_{i,j}'(x) 
      &=\begin{cases}
        x^{1/r-1} \left( y_{i,j}^{1-1/r}(x) - y_{i,j-1}^{1-1/r}(x) \right) ,& i< N, j< N, \\
        x^{1/r-1} \left( \dfrac{h_N}{r Z_1} - y_{i,N-1}^{1-1/r}(x) \right)  ,& i< N, j=N, \\
        x^{1/r-1} \left( (2T-y_{i,N+1}(x))^{1-1/r} - \dfrac{h_N}{r Z_1}  \right)   ,& i< N, j=N+1, \\
      \end{cases}
    \end{aligned}
  \end{equation*}
  and
  \begin{equation*}
    \begin{aligned}
      h_{i,j}'(x) 
      &=\begin{cases}
        x^{1/r-1} \left( (2T-y_{i,j}(x))^{1-1/r} - (2T-y_{i,j-1}(x))^{1-1/r} \right)   ,& i< N, j>N+1, \\
        \dfrac{rZ_1}{h_N} \left( y_{N,j}^{1-1/r}(x) - y_{N,j-1}^{1-1/r}(x) \right)  ,& i=N, j< N ,\\
        \dfrac{rZ_1}{h_N} \left( \dfrac{h_N}{rZ_1} - y_{N,N-1}^{1-1/r}(x) \right)   ,& i=N, j=N .
      \end{cases}
    \end{aligned}
  \end{equation*}
  
  For \(2\le i\le N\), \(2\le j < N\), it yields
  \begin{equation} \label{eq:yij1-1/r-j-1}
    \begin{aligned}
      y_{i,j}^{1-1/r}(\xi) - y_{i,j-1}^{1-1/r}(\xi)
      &\le x_{j+1}^{1-1/r} - x_{j-2}^{1-1/r}  \\
      & \le C T^{1-1/r} (r-1) N^{1-r} j^{r-2} = C  (r-1) Z_1 x_j^{1-2/r}.
    \end{aligned}
  \end{equation}
  
  For $2\le i\le N$, \(j=N\), since 
  \begin{equation} \label{eq:hN/rZ1}
    \dfrac{h_N}{r Z_1} = T^{1-1/r} \dfrac{1 - (1-h)^r}{rh} = \eta^{1-1/r} \simeq x_N^{1-1/r}, \quad \eta \in (x_{N-1}, x_{N}),
  \end{equation}
  we have
   $ |\frac{h_N}{r Z_1} - y_{i,N-1}^{1-1/r}(\xi)| \le x_N^{1-1/r} - x_{N-2}^{1-1/r} \simeq (r-1) Z_1 x_N^{1-2/r}$.
  
  For $2\le i\le N$, \(N+1\le j \le 2N-2\), it can be checked 
  $$|h_{i,j}'(\xi)| \le C (r-1)Z_1 \xi^{1/r-1} (2T-x_j)^{1-2/r}.$$
  Combine with \Cref{lmm:hilexi,lmm:gen-prop-of-MTFs}, the first inequality is obtained.

  On the other hand, from \eqref{def:yij}, we have $|y_{i,j}(x)-x| = \text{sign}(j-i) (y_{i,j}(x) - x)$ and $(y_{i,j}(x) - x)' = y_{i,j}'(x) - 1$.
  
  For \(2\le i < N\), \(2\le j<N\), by \Cref{ineq:a-b-theta,lmm:gen-prop-of-MTFs}, we have
  \begin{equation*}
    \xi^{1/r} |y_{i,j}^{1-1/r}(\xi) - \xi^{1-1/r}| \le |y_{i,j}(\xi) - \xi| \simeq |x_{j}-x_{i}|  .
  \end{equation*}
  
  For \(2\le i < N\), \(j=N\), using \eqref{eq:hN/rZ1} and \Cref{ineq:a-b-theta}, it yields
  \begin{equation} \label{eq:hN/rZ1-xi<N}
    \begin{aligned}
      \eta^{1/r} \left|\frac{h_N}{r Z_1} - \xi^{1-1/r}\right| &\le |\eta - \xi|, \quad \eta \in (x_{N-1}, x_{N})  \\
      & \le |x_N - x_i| + |h_N| + |h_{i+1}| \le 3 |x_N - x_i|.
    \end{aligned}
  \end{equation}
  
  For \(2\le i < N\), \(N<j\le 2N-2\), from \Cref{ineq:a-b-theta}, one has
  \begin{equation*} \label{eq:2T-yij1-1/r}
    \begin{aligned}
      &\xi^{1/r} |(2T-y_{i,j}(\xi))^{1-1/r} - \xi^{1-1/r}|
              \le |2T - y_{i,j}(\xi) - \xi|  \\
      & \quad \le |2T - x_j - x_i| + |y_{i,j}(\xi) - x_j| + |\xi - x_i|
              \le |2T - x_j - x_i| + 2 h_N   \\
      & \quad \le |x_j - T| + |T-x_i| + 2h_N \le 2 |x_j - x_i|.
    \end{aligned}
  \end{equation*}

  Similar to proof of \eqref{eq:hN/rZ1-xi<N}, we have \(y_{N,N}(x) - x \equiv 0\) and
  \begin{equation*}
  \begin{gathered}
    \eta^{1/r} |y_{N, j}^{1-1/r}(\xi) - \dfrac{h_N}{r Z_1}| \le C |x_j - x_N|, \quad i=N, j<N, \\
    \eta^{1/r} |(2T-y_{N, j}(\xi))^{1-1/r} - \frac{h_N}{r Z_1}| \le C |x_j - x_N|, \quad i=N, j>N.
  \end{gathered}
  \end{equation*}
  Thus, using \eqref{def:yij} and \Cref{lmm:gen-prop-of-MTFs}, the second inequality is obtained.
\end{proof}

\begin{proof} [\bf Proof of \Cref{lmm:esitmate-of-MTFs-2}]
  By \Cref{def:gridmapfunc}, \Cref{ineq:a-b-theta} and \eqref{eq:hN/rZ1}, there exist
  \begin{equation} \label{eq:Zj-i}
  \begin{split}
    &x_j^{1-1/r} |Z_{j-i}| = x_j^{1-1/r} |x_j^{1/r} - x_i^{1/r}| \le |x_j - x_i|, \quad i<N, j < N, \\
    &Z_{2N-j+i} \le Z_{2N} =  2 T^{1/r},                    \qquad i<N, j>N,    \\
    &\dfrac{h_N}{r Z_1} \simeq x_N^{1-1/r},      \qquad i=N, 2\le j \le 2N-2. 
  \end{split}
  \end{equation}
  Combine with \eqref{def:yij} and \Cref{lmm:gen-prop-of-MTFs}, the first inequality is obtained.

  From \eqref{def:hij}, it yields $h_{i,j}''(x) = y_{i,j}''(x) - y_{i,j-1}''(x)$.
  For $3\le j \le 2N-2$, we have
  \begin{equation}   \label{eq:y1-2/r}
  \begin{gathered}
  |y_{i,j}^{1-2/r}(\xi) - y_{i,j-1}^{1-2/r}(\xi)| \simeq (r-2) Z_1 x_j^{1-3/r},  \\
  |(2T-y_{i,j}(\xi))^{1-2/r} - (2T-y_{i,j-1}(\xi))^{1-2/r}| \simeq (r-2) Z_1 (2T-x_j)^{1-3/r},  \\
  \end{gathered}
  \end{equation}
  which can be similarly proven as \eqref{eq:yij1-1/r-j-1}. 
    
  For \(2\le i<N\), \(3\le j<N\), it yields
  \begin{equation*} \label{eq:yij1-2/rZj-i-depart}
    \begin{aligned}
      y_{i,j}^{1-2/r}(\xi) Z_{j-i} - y_{i,j-1}^{1-2/r}(\xi) Z_{j-i-1}
      = \left( y_{i,j}^{1-2/r}(\xi) - y_{i,j-1}^{1-2/r}(\xi) \right) Z_{j-i} + y_{i, j-1}^{1-2/r}(\xi) Z_1.
    \end{aligned}
  \end{equation*}
  Combine with \eqref{eq:Zj-i} and \eqref{eq:y1-2/r}, we get
  \begin{equation*} \label{eq:yij1-2/rZj-i-}
    |y_{i,j}^{1-2/r}(\xi) Z_{j-i} - y_{i,j-1}^{1-2/r}(\xi) Z_{j-i-1}| \le C Z_1 \left( |r-2| x_j^{-2/r}|x_j - x_i| + x_j^{1-2/r}  \right).
  \end{equation*}
  
  For $2\le i < N$, \(j=N, N+1\), it leads to
      $$|h_{i,j}''(x)| \le |y_{i,j}''(x)| + |y_{i,j-1}''(x)| \le C (r-1) x_i^{1/r-2} x_N^{1-1/r}.$$

  For $2\le i<N$, \(j > N+1\), from \Cref{lmm:hilexi} and \eqref{eq:y1-2/r}, we have
  \begin{equation*}
    \begin{aligned}
      & \left|  \delta(y_{i,j}(\xi))^{1-2/r} Z_{2N-(j-i)} - \delta(y_{i,j-1}(\xi))^{1-2/r} Z_{2N-(j-i-1)}  \right| \\
      &\le C Z_1 \left( |r-2| \delta(x_j)^{1-3/r} x_N^{1/r} + \delta(x_j)^{1-2/r}  \right) \le C Z_1 \delta(x_j)^{1-3/r} x_N^{1/r}.
    \end{aligned}
  \end{equation*}
  
  For \(i=N, j=N,N+1\), one has
  $  |h_{N,j}''(\xi)| 
    \le C x_N^{-1}$.
  For \(i=N\), $j\neq N, N+1$, using \eqref{eq:y1-2/r}, \eqref{def:yij} and \Cref{lmm:gen-prop-of-MTFs}, the second inequality is obtained.
\end{proof}

\begin{proof} [\bf Proof of \Cref{lmm:d2Pj-itle}]
  According to \(|y_{i,j}^\theta(\xi) - \xi| = \text{sign}(j-i-1+\theta) (y_{i,j}^\theta(\xi) - \xi)\) with $\theta\in(0,1)$, \Cref{lmm:Dh2simd2}, \eqref{def:I2-a} and \eqref{def:Pj-itheta-jlN}, we have
   $ D_h^2 {P_{i,j}^\theta}(x_i) = {P_{i,j}^\theta}''(\xi), \quad \xi\in (x_{i-1}, x_{i+1})$.
  From \Cref{lmm:gen-prop-of-MTFs,def:yij,lmm:esitmate-of-MTFs-1,lmm:esitmate-of-MTFs-2,lmm:regularity-u,lmm:hilexi}, and 
  regarding the selection process of $i,j$ within Case 1-3, it turns out that
  \begin{gather*}
    h_{i,j}(\xi) \le C h_j, \quad |h_{i,j}'(\xi)| \le C(r-1) h_j x_i^{-1}, \\
    |y_{i,j}^\theta(\xi) - \xi| \le C |y_j^\theta - x_i|, \quad 
    \left|(y_{i.j}^\theta(\xi) - x_i)'\right| \le C |y_j^\theta - x_i| x_i^{-1}, \\
    \left|u''(y_{i,j}^\theta(\xi))\right| \le C x_i^{\alpha/2-2}, \;
    \left|\left(u''(y_{i,j}^\theta(\xi))\right)'\right| \le C x_i^{\alpha/2-3}, \; 
    \left|\left(u''(y_{i,j}^\theta(\xi))\right)''\right| \le C x_i^{\alpha/2-4}.
  \end{gather*}
  By \Cref{lmm:esitmate-of-MTFs-2}, we have
  \begin{alignat*}{3}
    |h_{i,j}''(\xi)| &\le C(r-1) h_j x_i^{-2}, \quad
    &&\left|(y_{i,j}^\theta(\xi) - x_i)''\right| \le C(r-1) |y_j^\theta - x_i| x_i^{-2}, \quad&\text{for Case 1}, \\
    |h_{i,j}''(\xi)| &\le C(r-1), \quad    
    &&\left|(y_{i,j}^\theta(\xi) - x_i)''\right| \le C(r-1), \quad&\text{for Case 2},   \\
    |h_{i,j}''(\xi)| &\le C(r-1) h_j, \quad
    &&\left|(y_{i,j}^\theta(\xi) - x_i)''\right| \le C(r-1), \quad&\text{for Case 3}.
  \end{alignat*}
  Using Leibniz formula and chain rules, the desired results are obtained.
\end{proof}

\begin{proof} [\bf Proof of \Cref{lmm:dQj-itle}]
  \label{prf:dQj-itle}
  Since
  \begin{equation*}
    \begin{aligned}
       & \frac{{Q_{i,j,l}^\theta}(x_{i+1}) u'''(\eta_{j+1}^\theta) - {Q_{i,j,l}^\theta}(x_{i}) u'''(\eta_{j}^\theta)}{h_{i+1}}  \\
       & \quad = \frac{Q_{i,j,l}^\theta(x_{i+1}) - Q_{i,j,l}^\theta(x_{i})}{h_{i+1}} u'''(\eta_{j+1}^\theta)
      + Q_{i,j,l}^\theta(x_i) \frac{u'''(\eta_{j+1}^\theta)-u'''(\eta_{j}^\theta)}{h_{i+1}}    .
    \end{aligned}
  \end{equation*}
  There exists $\xi \in (x_{i}, x_{i+1})$ such that
    $\frac{Q_{i,j,l}^\theta(x_{i+1}) - Q_{i,j,l}^\theta(x_{i})}{h_{i+1}} = {Q_{i,j,l}^\theta}'(\xi)$.
  From \eqref{def:Qj-itheta-jlN}, \Cref{lmm:gen-prop-of-MTFs,lmm:esitmate-of-MTFs-1,lmm:hilexi}, Leibniz formula and chain rule, we have
  \begin{equation*}
    \begin{aligned}
      |{Q_{i,j,l}^\theta}'(\xi)| \le C h_j^l |y_{j}^\theta - x_{i}|^{1-\alpha} (x_i^{-1} + x_i^{1/r-1} \delta(x_j)^{-1/r}),\;\;
      Q_{i,j,l}^\theta(x_i) = C h_{j}^l |y_j^\theta-x_i|^{1-\alpha}.
    \end{aligned}
  \end{equation*}
  According to \Cref{lmm:hilexi,lmm:regularity-u}, it implies
   $ |u^{(l-1)}(\eta_{j+1}^\theta)| 
    \le C\delta(x_j)^{\alpha/2-l+1} $,
  and
  \begin{equation*}
    \begin{aligned}
      \frac{|u^{(l-1)}(\eta_{j+1}^\theta)-u^{(l-1)}(\eta_{j}^\theta)|}{h_{i+1}}
      &= |u^{(l)}(\eta)| \frac{\eta_{j+1}^\theta - \eta_{j}^\theta}{h_{i+1}}  , \quad \eta \in (x_{j-1}, x_{j+1})\\
      & \le C \delta(\eta)^{\alpha/2-l} \frac{x_{j+1}-x_{j-1}}{h_{i+1}}
       \simeq x_i^{1/r-1} \delta(x_j)^{\alpha/2-l+1-1/r} .
    \end{aligned}
  \end{equation*}
  Thus, the first inequality is obtained.
  The second one can be similarly proven. 
\end{proof}

\bibliographystyle{siamplain}
\bibliography{references}

\end{document}